\documentclass[12pt, a4paper]{amsart}
\usepackage[utf8]{inputenc}

\usepackage{lmodern}
\usepackage{upgreek}
\usepackage{color}
\usepackage{amsfonts}
\usepackage{amsmath}
\usepackage{amssymb}
\usepackage{amsthm}
\usepackage{etaremune}
\usepackage{enumitem}
\usepackage{pdfpages}
\usepackage{fancyhdr}
\usepackage{setspace}
\usepackage{hyperref}
\hypersetup{
	colorlinks=true,
	linkcolor=blue,
	filecolor=blue,      
	urlcolor=blue,
	citecolor=blue,
	pdftitle={},
	pdfpagemode=FullScreen}

\numberwithin{equation}{section}
\newtheorem{theorem}{Theorem}[section]
\newtheorem{lemma}[theorem]{Lemma}
\newtheorem{thm}{Theorem}[section]

\newtheorem{prop}[theorem]{Proposition}

\newtheorem*{conj}{Hardy-Littlewood $k$-tuple conjecture}

\theoremstyle{definition}
\newtheorem{definition}[theorem]{Definition}

\newtheorem{rem}[theorem]{Remark}
\newtheorem{nota}[theorem]{Notation}

\newcommand{\Mod}[1]{\ (\mathrm{mod}\ #1)}

\title{Almost primes and primes that are sums of two squares plus 1}

\author{Kunjakanan Nath}
\address{Institut \'Elie Cartan de Lorraine, Universit\'e de Lorraine, CNRS, F-54000 Nancy, France}
\email{kunjakanan@gmail.com}

\author{Likun Xie}
\address{Department of Mathematics, University of Illinois Urbana--Champaign, 1409 W.\ Green Street, Urbana, IL 61801, USA}
\curraddr{Max-Planck-Institut f\"ur Mathematik,
	Vivatsgasse 7, 53111  Bonn, Germany} 
\email{xie@mpim-bonn.mpg.de}

\date{\today}

\begin{document}
	
\begin{abstract}
		In this paper, we obtain a lower
		bound for the number of primes \( p \leq x \) such that \( p - 1 \) is a sum of two squares and \( p + 2 \) has a bounded number of prime factors. The proof uses the vector sieve framework, involving a semi-linear sieve and a linear sieve. 
\end{abstract}

	\subjclass[2020]{11N05, 11N35, 11N36}
	\keywords{linear sieve, semi-linear sieve, vector sieve}

	\maketitle

	\section{Introduction}\label{sec: Intro}
	
	One of the famous unsolved problems in number theory is the \emph{twin prime conjecture}, which states that there are infinitely many primes $p$ such that $p + 2$ is also a prime. In fact, Hardy and Littlewood made a more general conjecture. To state it, we need the notion of an admissible $k$-tuple.
	
	We say a set $\mathcal{H}=\{h_1 , h_2 , \dotsc , h_k\}$ is {\it admissible} if it avoids at least one residue class modulo $p$ for each prime $p$.
	
	\begin{conj}
		Let $\mathcal{H}=\{h_1, \dotsc, h_k\}$ be an admissible set. Then there are infinitely many integers $n$ such that $n + h_1, \dotsc, n+h_k$ are all primes. More precisely,
		\begin{align*}
			\#\{n\leq x\colon n+h_1, \dotsc, n+h_k\: \text{are primes}\} \sim \mathfrak{S}(k)\int_{2}^x\dfrac{dt}{(\log t)^k} \quad \text{as $x\to \infty$},
		\end{align*}
		where
		\begin{align*}
			\mathfrak{S}(k):=\prod_{p\: \text{prime}}\bigg(1-\dfrac{\#(\mathcal{H}\Mod p)}{p}\bigg)\bigg(1-\dfrac{1}{p}\bigg)^{-k}.
		\end{align*}
	\end{conj}
	
	Although the above conjecture is far from being solved, there has been spectacular progress recently. Using sieve methods (for example, see \cite[Theorem 5.8]{HR1974}), one can show that
	\[ \#\{n\leq x\colon n+h_1, \dotsc, n+h_k\: \text{are primes}\}\ll_{k}\mathfrak{S}(k)\dfrac{x}{(\log x)^k}.\]
	For the lower bound, if $\mathcal{H}=\{0, 2\}$, the celebrated work of Chen \cite{Che1973} implies that
	\[
	\#\{p \leq x \colon p \text{ is prime and } p + 2 \text{ has at most two prime factors}\} \gg \frac{x}{(\log x)^2}.
	\]
	In a different direction, assuming the Elliot-Halberstam conjecture, Goldston, Pintz, and Y{\i}ld{\i}r{\i}m \cite{GPY2009} showed the bounded gaps between primes. However, the recent works of Zhang \cite{Zha2014}, Maynard \cite{May2015}, and Polymath \cite{Poly} imply the existence of bounded gaps between primes unconditionally.
	
	More recently, Heath-Brown and Li \cite{HBL2016} considered the generalization of Chen's theorem for prime triples. They showed that there are infinitely many primes $p$ such that $p+2$ has at most two prime factors and \( p + 6 \) has at most \( 76 \) prime factors. 
	More precisely, they showed that 
	\begin{align*}
		\#\{p\leq x\colon p\: \text{prime}, \:  \Omega(p+2)\leq 2, \Omega(p+6)\leq 76\}\gg\dfrac{x}{(\log x)^3},
	\end{align*}
	where $\Omega(n)$ counts the number of prime factors of $n$ counted with multiplicity. The bound \( 76 \) has been reduced to \( 14 \) by Cai \cite{Cai2017}, and Zhu \cite{Zhu2024} further improved it to \( 11 \).
	
	We also take this opportunity to mention Friedlander and Iwaniec's \cite{FI2009} work towards the {Hyperbolic Prime Number Theorem}. Assuming the Elliott-Halberstam Conjecture, they\footnote{The upper bound is established without any assumption.} showed that
	\[x\ll \sum_{n\leq x}\Lambda(n)r(n-2)r(n+2)\ll x,\]
	where $\Lambda$ is the von-Mangoldt function and $r(n)$ counts the number of representation of $n$ as the sum of two squares. The above result reflects the correlation between primes $p$ and the sums of two squares of the form $p-2$ and $p+2$.
	
	For more general sieve-theoretic results towards prime $k$-tuple conjecture, we invite the readers to see \cite{FI2010, HR1974}. There has also been some progress on the twin prime conjecture on average using analytic methods; see recent work of Matom\"aki, Radziwi\l{}\l{}, and Tao \cite{MRT2019i}.

	An interesting subset of primes is the primes of the form $p=m^2+n^2+1$, where $m$ and $n$ are some non-zero integers. It is one of the simplest non-trivial examples of a ``sparse subset of the primes'' consisting of the values of a multivariate polynomial. Indeed, an application of the upper bound sieve (see \cite{Iwa1972} or \cite{Mot1971}) shows that
	\begin{align}
		\notag \#\{p\leq x\colon p=m^2+n^2+1, \: p~\text{prime}\}\ll\dfrac{x}{(\log x)^{3/2}}.
	\end{align}
	It is also known that there are infinitely many primes of the form $p=m^2+n^2+1$, a result due to Linnik \cite{Lin1960}, who established it by using the \emph{dispersion method}. For the lower bound, Motohashi \cite{Mot1970} showed that
	\begin{align}
		\notag \#\{p\leq x\colon p=m^2+n^2+1, p~\text{prime}\}\gg \dfrac{x}{(\log x)^{2}}.
	\end{align}
	Finally, in his celebrated work \cite{Iwa1972}, Iwaniec used the semi-linear sieve to establish the matching lower bound 
	\begin{align}
		\notag \#\{p\leq x\colon p=m^2+n^2+1, \: p~\text{prime}\}\gg \dfrac{x}{(\log x)^{3/2}}. 
	\end{align}
	We still do not have an asymptotic formula for the number of primes of the form $p=m^2+n^2+1$ up to any positive real number $x$ unconditionally. See \cite[Corollary 2]{Iwa1976} for a conditional asymptotic formula. There have also been several generalizations of Iwaniec's result to other subsets of integers; for example, see Huxley-Iwaniec \cite{HI1975}, Wu \cite{Wu1998}, Matom\"aki \cite{Mat2007}, Ter\"av\"ainen \cite{Ter2018}, Nath \cite{Nat2024}.

	\subsection{Main result} It is reasonable to expect that there are infinitely many primes of the form $p=m^2+n^2+1$ such that $p+2$ is also a prime. In fact, it is expected that as $x\to \infty$,
    \[\{p\leq x\colon p\: \text{prime}, p=1+m^2+n^2, m, n\in \mathbb{N}, p+2\:\text{prime}\}\sim C\dfrac{x}{(\log x)^{5/2}}\]
    for some constant $C>0$.
    
    Motivated by the above expectation, we will investigate the correlation between primes of the form \( p = m^2 + n^2 +1 \) and almost primes\footnote{Given any integer $k\geq 1$, we say a positive integer $n$ is $k$-almost prime if $n$ has at most $k$ prime factors.} of the form \( p + 2 \). We will show that there are infinitely many primes of the form $p=m^2+n^2+1$ such that $p+2$ has at most $9$ prime factors.

	\begin{thm}\label{theorem1.1}
		We have
		\[
		\#\left\{ 
		p \leq x\colon  p \text{ prime}, \, p = m^2 + n^2+1, \, m, n \in \mathbb{N}, \, \Omega(p+2) \leq 9
		\right\} \gg \frac{x}{(\log x)^{5/2}}.
		\]
	\end{thm}

    \begin{rem}
        Assuming a version of the Elliot-Halberstam conjecture\footnote{By a version of the Elliot-Halberstam conjecture, we mean that for any $\varepsilon>0$, the relation \eqref{27} in Lemma \ref{Lemma_HBL_2} holds for $q\leq x^{1-\varepsilon}$.} and then establishing a variant of Propositions \ref{Prop: lower bound} and \ref{Prop: weighted}, it is possible to show that
        \[
		\#\left\{ 
		p \leq x\colon  p \text{ prime}, \, p = m^2 + n^2+1, \, m, n \in \mathbb{N}, \, \Omega(p+2) \leq 3
		\right\} \gg \frac{x}{(\log x)^{5/2}}.
		\]
    \end{rem}

	\begin{rem}
		One of the key components in the proof of Theorem~\ref{theorem1.1} is the application of vector sieves of mixed dimensions (see Proposition \ref{vector_sieve}). First, we apply the vector sieve with a semi-linear sieve and a linear sieve in conjunction with Iwaniec's argument to detect primes of the form $p=m^2+n^2+1$. Secondly, we incorporate weighted sieves to optimize the prime factors of $p+2$. For more details, we would invite the readers to see Section \ref{sec: outline}.
	\end{rem}

	\begin{rem}
		Beyond the case of almost primes of the form \( p+2 \), our approach could also be extended to study primes of the form \( p =  m^2 + n^2+1 \) such that \(\prod_{i=1}^k (p + a_i) \) is an almost prime. For instance, the vector sieve described in Proposition~\ref{vector_sieve} can be employed for this generalization, combining a semi-linear sieve with a higher-dimensional sieve. The strategy closely parallels the proof of Theorem~\ref{theorem1.1}. While we highlight this potential application, we do not pursue further details here. For example, the proof of Theorem~\ref{theorem1.1} can be followed verbatim to yield a result of the form
		\[
		\#\left\{ 
		p \leq x\colon p\: \text{prime},\:   p =  m^2 + n^2+1, \, m, n \in \mathbb{N}, \, \Omega((p+2)(p+6)) \leq r 
		\right\} \gg  \frac{x}{(\log x)^{5/2}},
		\]
		for some positive integer \( r \) which can be optimized through numerical computation.
		
	\end{rem}

	\begin{rem}
		It would also be of interest to extend the vector sieves to multiple variables in Proposition~\ref{vector_sieve}. This would enable the study of primes \( p = m^2 + n^2+1 \) such that \( p + a_1, p + a_2, \dotsc, p+a_k\) are almost primes. Such an extension could lead to results of the form:
		\[
		\#\left\{ 
		p \leq x\colon p: \text{prime}, \: p =  m^2 + n^2+1, \, m, n \in \mathbb{N}, \, \Omega(p+2) \leq r, \, \Omega(p+6) \leq s  
		\right\} \gg \frac{x}{(\log x)^{7/2}},
		\]
		for some positive integers \( r, s \) that can be optimized through numerical computations. While we do not pursue the details of these directions here, we highlight the potential for further extensions of vector sieves.
	\end{rem}

	\subsection{Notation}
	
	We will use standard notation throughout the paper. The set $\mathbb{N}=\{1, 2, \dotsc\}$ denotes the set of natural numbers unless specified otherwise, and $m$, $n$ always denote natural numbers. We reserve the letters $p, p_1, p_2, p_3, \dotsc, p_r$ to denote primes.
	
	Given functions $f, g\colon \mathbb{R}\to \mathbb{C}$, the expressions of the form $f(x)=O(g(x))$, $f(x) \ll g(x)$, and $g(x) \gg f(x)$ signify that $|f(x)| \leq c|g(x)|$ for all sufficiently large $x$, where $c>0$ is an absolute constant. A subscript of the form $\ll_A$ means the implied constant may depend on the parameter $A$. The notation $f(x) \asymp g(x)$ indicates that $f(x) \ll g(x) \ll f(x)$. We also let $o(1)$ denote a quantity that tends to zero as $x \rightarrow \infty$. Finally, $f(x) \sim g(x)$ means $f(x)$ and $g(x)$ are asymptotatically equivalent, that is, $\lim_{x \rightarrow \infty}f(x)/g(x)=1$. All the above asymptotic notation should be interpreted as referring to the limit $x \rightarrow \infty$.
	
	For any set $\mathcal{S}$, $\#\mathcal{S}$ denotes the cardinality of the set $\mathcal{S}$. We let $1_{\mathcal{S}}$ be the characteristic function of the set $\mathcal{S}$ (so $1_{\mathcal{S}} (x) = 1$ if $x\in \mathcal{S}$ and $0$, otherwise).
	
	We write \( \varphi \) to denote the Euler totient function. We let $\Omega(n)$ and $\omega(n)$ to denote the number of prime factors of $n$ counted with multiplicity and without multiplicity, respectively. 
	
	On the other hand, \( \tau_k(n) \) will represent the number of representations of \( n \) as a product of \( k \) factors. When \( k = 2 \), we will write \( \tau = \tau_2 \) to denote the divisor function.
	
	For any two arithmetic functions $f, g\colon \mathbb{N}\to \mathbb{C}$, we write $(f*g)(n):= \sum_{ab=n}f(a)g(b)$ for their Dirichlet convolution.

    Finally, we set $(a, b)$ to be the greatest common divisor of integers $a$ and $b$.

	\subsection{Organization of the paper}
	The remaining sections of the paper are organized
	as follows. In Section \ref{sec: outline}, we give an outline of the proof of the theorem. We use Section \ref{sec: preliminaries} to gather several preliminary results on the Bombieri-Vinogradov type theorems, beta sieves, and vector sieves. In Section \ref{sec: sieve estimates}, we will use ingredients from Section \ref{sec: preliminaries} to establish the key sieve estimates, namely, Propositions \ref{Prop: lower bound}, \ref{Prop: upper bound}, and \ref{Prop: weighted}. Finally, we will give the proof of Theorem \ref{theorem1.1} by combining sieve estimates with the numerical calculations in Section \ref{sec: proof}.

   \subsection*{Acknowledgements} The authors would like to thank Bruce Berndt and C\'ecile Dartyge for their helpful comments and corrections on an earlier version of the paper. KN would like to thank Kevin Ford, Dimitris Koukoulopoulos, and Youness Lamzouri for their guidance and encouragement.

	\section{Outline of the proof and the setup}\label{sec: outline}

	Let $\theta_2\in (0, 1/2)$ be a parameter to be chosen appropriately. Consider the following setup
	\begin{align*}
		\mathcal{A}=\bigg\{p-1\colon p\leq x-2, \: p\equiv 3\Mod 8, \: (p+2, P(x^{\theta_2}))=1\bigg\} \quad \text{and} \quad \mathcal{Q}=\{p\equiv 3\Mod 4\},
	\end{align*}
	where $P(w)=\prod_{p<w}p$ for any $w\geq 2$.  For any $z> 3$, set
	\begin{align*}
		Q(z):=\prod_{\substack{p<z\\p\in \mathcal{Q}}}p
	\end{align*}
	and for $d|Q(z)$, set
	\begin{align*}
		\mathcal{A}_d:=\sum_{\substack{n\in \mathcal{A}\\ d|n}}1.
	\end{align*}
	Define the sifting function as
	\begin{align}\label{Def: sifting semi}
		S(\mathcal{A}, \mathcal{Q}, z)=\#\{n\in \mathcal{A}\colon (n, Q(z))=1\}.
	\end{align}
	Then, by Buchstab's identity, we have
	\begin{align}\label{Eq: App of Buchstab}
		{S}(\mathcal{A}, \mathcal{Q}, x^{1/2})={S}(\mathcal{A}, \mathcal{Q}, x^{\theta_1})-\sum_{\substack{x^{\theta_1}<p_1\leq \sqrt{x}\\ p_1\in \mathcal{Q}}}S(\mathcal{A}_{p_1}, \mathcal{Q}, p_1),
	\end{align}
	where $0<\theta_2<\theta_1<1/2$ and $\theta_1\asymp \theta_2$ is a parameter to be chosen appropriately.
	Note that
	\[
	S(\mathcal{A}, \mathcal{Q}, x^{1/2}) = \#\mathcal{A}^{(0)},
	\]  
	where \( \mathcal{A}^{(0)} := \{ n \in \mathcal{A}\colon n \text{ has no prime factor congruent to}\: 3 \Mod 4 \}. \) This implies that if $n\in \mathcal{A}^{(0)}$, then $n=p-1$, with $p\equiv 3\Mod 8$, $n$ can be written as the sum of two squares and $n+3$ has at most $r:=\lfloor \theta_2^{-1}\rfloor$ prime factors. This will establish our theorem with $p+2$ having at most $r$ prime factors. However, we will use the weighted sieve to improve the value of $r$. To apply the weighted sieve, we begin by defining the sets 
	\begin{align}\label{Def: set B}
		\mathcal{B} := \{ p+2\colon p-1 \in \mathcal{A}^{(0)}\} \quad \text{and} \quad \mathcal{B}_p=\sum_{\substack{b\in \mathcal{B}\\p|b}}1 .
	\end{align}
	Next, we introduce a weight function\footnote{This type of weight function is often referred to as Richert's weight in the literature.}
	\[
	w_p = 1 - \frac{\log p}{\log y},
	\]  
	where $y=x^\theta$ with $\theta_2<\theta<1$. Note that any element of $\mathcal{B}$ is coprime to $P(x^{\theta_2})$ and $w_p<0$ for $p>y$. So, we have  
	\[
	\begin{aligned}
		\sum_{x^{\theta_2} \leq p \leq y} w_p \#\mathcal{B}_p &\geq \sum_{3 \leq p \leq x} w_p \#\mathcal{B}_p \\
		&= \sum_{b \in \mathcal{B}} \bigg(\omega(b) - \frac{1}{\log y} \sum_{p \mid b} \log p \bigg) \\
		&\geq \sum_{b \in \mathcal{B}} \bigg(\omega(b) - \frac{\log x}{\log y}\bigg)
		= \sum_{b \in \mathcal{B}}\big(\omega(b) - \theta^{-1}\big),
	\end{aligned}
	\]
	where we have used the fact that $y=x^\theta$ in the last line. Define $z_2=x^{\theta_2}$. Therefore, for any $\lambda>0$, we have
	\begin{align}
		S(\mathcal{A}, \mathcal{Q}, x^{1/2}) 
		& - \lambda \sum_{z_2 \leq p \leq y} w_p \#\mathcal{B}_p\nonumber \\
		& = \#\mathcal{B} - \lambda \sum_{z_2 \leq p \leq y} w_p \#\mathcal{B}_p\nonumber \\
		& \leq \sum_{b\in \mathcal{B} }\left(1 + \lambda/ \theta -\lambda\omega(b)  \right) \nonumber \\
		& \leq \left(1 + \lambda/ \theta\right) \#\big\{b \in \mathcal{B}\colon \omega(b) < 1/\lambda+1/\theta\big\} \nonumber \\
		& \leq \left(1 + \lambda/ \theta\right) \#\big\{b \in \mathcal{B}\colon \omega(b) < 1/\lambda+1/\theta, \, b \text{ square-free}\big\} + O\left(\frac{x}{z_2}\right) \label{number_of_divisors_p+2},
	\end{align}
	where the last inequality comes from the fact that the number of elements of \( \mathcal{B} \) which are not square-free is $\ll x^{1-\theta_2}$ as any element in $\mathcal{B}$ is coprime to $P(x^{\theta_2})$ by our assumption.
	So, our task reduces to showing that there exists a positive constant $c$ 
	such that
	\begin{align}\label{Eq: Main estimate}
		S(\mathcal{A}, \mathcal{Q}, x^{1/2}) - \lambda \sum_{x^{\theta_2} \leq p \leq y} w_p \#\mathcal{B}_p \geq \left(c + o(1)\right) \frac{x}{(\log x)^{5/2}}.
	\end{align}
	This will imply that 
	\begin{align*}
		\#\big\{b \in \mathcal{B}\colon \omega(b)<1/\lambda + 1/\theta, \, b \text{ square-free}\big\} \gg \dfrac{x}{(\log x)^{5/2}}.
	\end{align*}
	By the definition of $\mathcal{B}$, the left-hand side of the above expression counts the number of primes $p\leq x-2$ such that $p\equiv 3\Mod 8$, $p-1$ is the sum of two squares, $(p+2, P(x^{\theta_2}))=1$ and $\omega(p+2)\leq 1/\lambda+1/\theta$, which gives our desired result. So, the goal is to optimize the values of $\lambda$ and $\theta$.
	
	One of the key difficulties in establishing \eqref{Eq: Main estimate} is that we have to take care of two simultaneous sieving conditions $(p-1, Q(x^{\theta_1}))=(p+2, P(x^{\theta_2}))=1$, where $\theta_1\asymp \theta_2$. It would have been much easier to deal with if we knew that primes $p\leq x$ satisfying either $(p-1, Q(x^{\theta_1}))=1$ or $(p+2, P(x^{\theta_2}))=1$ are well-distributed in arithmetic progressions. So, to circumvent this issue, we apply the vector sieve, which consists of a semi-linear sieve and a linear sieve. In particular, we will obtain a fundamental lemma for vector sieves comprising two beta sieves in Proposition \ref{vector_sieve}, which is similar to \cite[Section 3.3]{HBL2016}. Moreover, to estimate the sum \[\sum_{\substack{x^{\theta_1}<p_1\leq \sqrt{x}\\ p_1\in \mathcal{Q}}}S(\mathcal{A}_{p_1}, \mathcal{Q}, p_1),\]
    we will use a ``switching principle'' due to Iwaniec \cite{Iwa1972} before applying the upper bound sieve.

	\section{Preliminaries}\label{sec: preliminaries}
	
	\subsection{Bombieri--Vinogradov type theorems}
	
	We will require a generalized version of the Bombieri-Vinogradov theorem established by Pan \cite{Pan1981}.
	
	\begin{theorem}[\cite{Pan1981}, Theorem 3]\label{theorem_pan}
		Let $x\geq 2$. For integers $m, q\geq 1$ and $a$, define
		\begin{align}\label{Def: generalized prime counting function}
			\pi(x; m, q, a):  = \sum_{\substack{mp \leq x\\ p\: \text{prime}\\ mp \equiv a \Mod q}} 1.
		\end{align}
		Let $f\colon \mathbb{N}\to \mathbb{R}$ be such that $f(m)=O(1)$ for all $m\in \mathbb{N}$. Then, for any given \( A > 0 \) and for any $\varepsilon>0$, we have
		\[
		\sum_{q \leq x^{1/2} (\log x)^{-B}} \max_{y \leq x} 
		\max_{\substack{(a, q) = 1}} 
		\bigg| \sum_{\substack{m \leq x^{1-\varepsilon} \\ (m,\: q) = 1}} 
		f(m) \bigg( \pi(y; m, q, a) - \frac{\pi(y; m, 1, 1)}{\varphi(q)} \bigg) \bigg| 
		\ll \dfrac{x}{(\log x)^A},
		\]
		where \( B = 3A/2 + 17 \). Here, the implied constant may depend only on $A$ and $\varepsilon$.
	\end{theorem}
	
	Next, we will use the following version of the Bombieri--Vinogradov theorem in our estimates.

	\begin{lemma}\cite[Lemma 2]{HBL2016}\label{Lemma_HBL_2}
		Let $x\geq 2$. For $z_1, z_2, \dots, z_r \geq 2$, define the set with multiplicities,
		\[
		P(z_1, \dots, z_r, y_1, \dots, y_r) = \{p^{(r)} = p_1 \cdots p_r : y_1 \geq p_1 \geq z_1, \dots, y_r \geq p_r \geq z_r\}.
		\]
		Let $\pi_r(x; q, a)$ be the number of $p^{(r)} \in P(z_1, \dots, z_r)$ such that $p^{(r)} \equiv a \pmod{q}$ and $p^{(r)} \leq x$. For each $q \geq 1$, let
		\begin{equation}\label{26}
			R_q(x) = \max_{\substack{(a, q) = 1}} \bigg| \pi_r(x; q, a) - \frac{1}{\phi(q)} \pi_r(x; q) \bigg|.
		\end{equation}
		Then for any $A > 0$ and $k \geq 1$, there exists $B = B(A, k) > 0$ such that
		\begin{equation}\label{27}
			\sum_{q < x^{1/2}(\log x)^{-B}}\tau(q)^k R_q(x) \ll \dfrac{x}{(\log x)^{A}},
		\end{equation}
		where the implied constant depends only on $r$, $k$, and $A$.
	\end{lemma}

	\subsection{The beta sieve}\label{section beta sieve}
	
	For the convenience of the readers, we begin with the definitions of an {\it upper bound sieve} and a {\it lower bound sieve}. 
	
	Given a set of primes $\mathcal{P}$ and a parameter $z\geq 2$, let
	\begin{align*}
		\mathcal{P}(z):=\prod_{\substack{p<z\\ p\in \mathcal{P}}}p.
	\end{align*}
	
	\begin{definition}[Upper bound sieve]
		An arithmetic function $\lambda^+:\mathbb{N}\rightarrow\mathbb{R}$ that is supported on the set $\{d|\mathcal{P}(z)\colon \: d \leq D\}$ and satisfies the relation
		\[(\lambda^+*1)(n) \geq 1_{(n,\: \mathcal{P}(z))=1},\]
		is called an \emph{upper bound sieve of level $D$} for the set of primes $\mathcal{P}$.
	\end{definition}
	
	\begin{definition}[Lower bound sieve]
		An arithmetic function $\lambda^-:\mathbb{N}\rightarrow\mathbb{R}$ that is supported on the set $\{d|\mathcal{P}(z)\colon  \: d \leq D\}$ and satisfies the relation
		\[1_{(n,\: \mathcal{P}(z))=1} \geq (\lambda^-*1)(n),\]
		is called a \emph{lower bound sieve of level $D$} for the set of primes $\mathcal{P}$.
	\end{definition}

    \begin{nota}
		We will refer to $\Lambda^\pm=(\lambda^{\pm})$ as the \emph{sieve weights} or \emph{sifting weights} in this paper. 
	\end{nota}

	A \textit{beta sieve} is a combinatorial sieve where the sieve weights are given by the M\"obius function, restricted to certain specific subsets of positive integers (see \cite[(11.17)--(11.18)]{FI2010}). For more details on the beta sieve, we invite the readers to see Chapter 11 of Friedlander and Iwaniec \cite{FI2010}.

	For $D_1, D_2 \geq 1$, let $\Lambda_i^\pm(D_i) = (\lambda_i^\pm(d))_{d \leq D_i}$ be two beta sieve weights of level $D_1$, $D_2$, respectively, and of dimensions $\kappa_1, \kappa_2>0$, respectively. In the following, we define the corresponding parameters, sieving sets, and functions for these two sieves, where the subscript $i$ denotes the objects associated with the sieve $\Lambda_i^\pm(D_i)$. 
	
	Let $\mathcal{P}_1$ and $\mathcal{P}_2$ be two given sets of primes. For $z_1, z_2\geq 2$, define
	\begin{align*}
		P_i(z_i)=\prod_{\substack{p<z_i\\ p\in \mathcal{P}_i}}p.
	\end{align*}

	Next, for $i\in \{1, 2\}$, we define
	\[
	V_i(z) =\sum_{d|P_i(z)}\mu(d)h_i(d)=\prod_{\substack{p<z\\p\in \mathcal{P}_i}} (1-h_i(p)) ,
	\]
	and 
	\begin{align*}
		V_i^{\pm } (D, z) =\sum_{d\mid P_i(z) } \lambda_i^{\pm }(d)  h_i(d) \quad \text{for $i\in \{1, 2\}$},
	\end{align*}
	where $h_i\colon \mathbb{N}\to \mathbb{R}$ is a multiplicative function such that
	\[
	0 \leq h_i(p) < 1 \quad \text{for all primes $p\in \mathcal{P}_i$}.
	\]
	We also assume throughout that for every \( 2 \leq w < z \), there exist constants \( L_i\geq 1 \) such that
	\begin{equation}\label{Condition: kappa}
		\dfrac{V_i(w)}{V_i(z)}\leq \left( \frac{\log z}{\log w} \right)^{\kappa_i} \left(1 + \frac{L_i}{\log w}\right).
	\end{equation}

	With the above preparations, we can now state the following result on the beta sieve from \cite[Theorem 11.12]{FI2010}.

	\begin{lemma}\label{Lem: Beta sieve}
		Let $\mathcal{P}_1$ and $\mathcal{P}_2$ be given sets of primes and let $D_1, D_2\geq 1$, $z_1, z_2\geq 2$. Suppose $L_1, L_2\geq 1, \kappa_1, \kappa_2>0$ are such that the relation \eqref{Condition: kappa} holds. Let  $\Lambda_1^\pm(D_1) =\{\lambda_1^\pm(d)\}_{d\leq D_1}$ and $\Lambda_2^\pm(D_2)=\{\lambda_2^\pm(d)\}_{d\leq D_2}$ be two beta sieve weights. Then for $i\in \{1, 2\}$, we have
		\begin{align}
			V_i^+(D_i,z_i) &\leq \big(F_i(s_i)  +O(\log D_i)^{-\frac{1}{6}}) \big) V_i(z_i) \text{ if } s_i\geq \beta_i-1 ,\\
			V_i^-(D_i,z_i) &\geq \big(f_i(s_i)  +O(\log D_i)^{-\frac{1}{6}} )\big) V_i(z_i) \text{ if } s_i\geq \beta_i, 
		\end{align}
		where 
		\[s_i=\dfrac{\log D_i}{\log z_i}\]
		and $F_i, f_i$ are continuous functions in $s_i$ satisfying system of differential-difference equations
		\begin{align*}
			\begin{cases}
				s_i^{\kappa_i}F_i(s_i)=A_i & \text{if $\beta_i-1\leq s_i\leq \beta_i+1$},\\
				s_i^{\kappa_i}f_i(s_i)=B_i & \text{at $s_i=\beta_i$},\\
			\end{cases}
		\end{align*}
		\begin{align*}
			\begin{cases}
				\dfrac{d}{ds_i}(s_i^{\kappa_i}F_i(s_i)) =\kappa_i s_i^{\kappa_i-1}f_i(s_i-1) & \text{if $s_i>\beta_i-1$},\\
				\dfrac{d}{ds_i}(s_i^{\kappa_i}f_i(s_i)) = \kappa_i s_i^{\kappa_i-1}F_i(s_i-1) & \text{if $s_i>\beta_i$},
			\end{cases}
		\end{align*}
		and $\beta_i=\beta_i(\kappa_i)$, $A_i=A_i(\kappa_i)$ and $B_i=B_i(\kappa_i)$ are explicit constants defined in \cite[(11.55)--(11.63)]{FI2010}.
	\end{lemma}
	
	\begin{rem}
		We will apply the above lemma with $\kappa_i\in \{1/2, 1\}$. In particular, by \cite[p. 225]{FI2010} if $\kappa_i=1/2$, then $\beta_i(1/2)=1$ and if $\kappa_i=1$, then $\beta_i(1)=2$.
	\end{rem}

    \begin{rem}
        If $\kappa_i=1/2$, we call the associated beta sieve weights the \emph{semi-linear sieve} weights. On the other hand, if $\kappa_i=1$, the beta sieve weights are called \emph{linear sieve} weights.
    \end{rem}

	\subsection{The vector sieves for two beta sieves}
	
	Our aim is to combine two beta sieves within the framework of a vector sieve. Our approach is similar to \cite[Section 3.3]{HBL2016}, but it involves a different preliminary sieving procedure, as the two sifting sets of primes are not necessarily identical.
	
	We begin with the following setup. Let \( \mathcal{W} \) be a finite subset of \( \mathbb{N}^2 \). Suppose \( z_1, z_2 \geq 2 \) with
	\[
	\log z_1 \asymp \log z_2,
	\]
	and write \( \mathbf{z} = (z_1, z_2) \). Given any two sets of primes $\mathcal{P}_1$ and $\mathcal{P}_2$, our goal is to estimate the following quantity
	\[
	S(\mathcal{W}; \mathbf{z}) = \#\{ (m, n) \in \mathcal{W}\colon (m, P_1(z_1)) = (n, P_2(z_2)) = 1 \},
	\]
	where
	\begin{align*}
		P_1(z_1)=\prod_{\substack{p<z_1\\ p\in \mathcal{P}_1}}p \quad \text{and} \quad  P_2(z_2)=\prod_{\substack{p<z_2\\ p\in \mathcal{P}_2}}p.
	\end{align*}
	To be precise, we wish to derive both upper and lower bounds for \( S(\mathcal{W}; \mathbf{z}) \) using the bounds for the two beta sieves described in Subsection \ref{section beta sieve}.
	
	In order to proceed further, if \( \mathbf{d} = (d_1, d_2) \) and \( \mathbf{n} = (n_1, n_2) \), we write \( \mathbf{d}|\mathbf{n} \) to denote $d_1|n_1$ and $d_2|n_2$. Define
	\begin{equation}\label{def_W_d}
		\mathcal{W}_{\mathbf{d}} = \{ \mathbf{n} \in \mathcal{W} : \mathbf{d}| \mathbf{n} \}.
	\end{equation}
	We impose the following axioms:
	
	\begin{enumerate}[label=(A\arabic*)]
		\item \label{axiom 1} There exists a multiplicative function $h\colon \mathbb{N}^2\to (0,1]$ such that
		\[
		\# \mathcal{W}_{\mathbf{d}}  = h(\mathbf{d}) X + r(\mathbf{d}),
		\]
		where $X$ can be interpreted as an approximation to $\#\mathcal{W}$ and $r(\mathbf{d})$ is a real number, which we think of as an error term.
		\item \label{axiom 2} For all primes \( p \), we have
		\[
		h(p,1) + h(1, p) - 1 < h(p, p) \leq h(p,1) + h(1, p),
		\]
		with bounds
		\begin{equation}\label{condition_h(d)}
			h(p,1), h(1, p) \leq \frac{C_1}{p} \quad \text{and} \quad h(p, p) \leq \frac{C_1}{p^2}
		\end{equation}
		for some constant \( C_1 \geq 2 \). 
		
		\item \label{axiom 3} For any integer $d\geq 1$, let  
		\[
		h_1(d) := h(d,1) \quad \text{and} \quad h_2(d) := h(1,d),
		\]
		where the functions $h_1, h_2\colon \mathbb{N}\to (0,1]$ satisfy the relation \eqref{Condition: kappa} for some positive constants \( \kappa_1, L_1 \) and \( \kappa_2, L_2 \), respectively.  
		
	\end{enumerate}

	Finally, for real numbers $z_0<z_1, z_2$, we define
	\begin{align}\label{definition_Vi}
		V_i(z_i, z_0) := \prod_{\substack{z_0 \leq p < z_i \\ p \in \mathcal{P}_i}} \left( 1 - h_i(p) \right), \quad (i = 1, 2) \quad \text{and} \quad  V(z_0, h^*):=\prod_{p<z_0}\big(1-h^*(p)),
	\end{align}
	where $h^*$ is a multiplicative function given by
	\begin{align} \label{Def: hstar}
		h^*(d)=\sum_{\substack{e_1e_2e_3=d\\ e_1e_3|P_1(z_0)\\ e_2e_3|P_2(z_0)}}h(e_1e_3, e_2e_3)\mu(e_3).
	\end{align}
	
	\begin{rem}
		Note that given any prime $p$, $h^*(p)$ is non-zero if and only if $p\in \mathcal{P}_1\cup \mathcal{P}_2$ and $p<z_0$.
	\end{rem}
	With the above preparation, we are now ready to state the fundamental lemma for the vector sieve.

	\begin{prop}\label{vector_sieve}
		Let $\mathcal{P}_1$ and $\mathcal{P}_2$ be two sets of primes. Let $\mathcal{W}$ be a finite subset of $\mathbb{N}^2$ satisfying axioms \ref{axiom 1}, \ref{axiom 2}, and \ref{axiom 3}. Assume that there exist $\beta_1, \beta_2\geq 1$ as in Lemma \ref{Lem: Beta sieve}. Let \( D = z_1^{s_1} z_2^{s_2} \), where \( \beta_i - 1 \leq s_i \ll 1 \) for $i\in \{1, 2\}$, \( \log z_1 \asymp \log z_2 \), and  $z_0 = e^{(\log z_1 z_2)^{1/3}}$. Furthermore, we define
		\[
		\sigma_i = \frac{\log D}{\log z_i}, \quad i\in \{1, 2\}.
		\]
		Then, if $\mathbf{z}=(z_1, z_2)$, we have
		\begin{align}\label{upperbound}
			S(\mathcal{W}; \mathbf{z}) &\leq XV(z_0, h^*) V_1(z_1, z_0) V_2(z_2, z_0) F(\sigma_1, \sigma_2)\left(1 + O\left((\log D)^{-\frac{1}{6}}\right)\right) + R^+(z_0, z_1, z_2),
		\end{align}
		and 
		\begin{align}\label{lowerbound}
			S(\mathcal{W}; \mathbf{z}) &\geq XV(z_0, h^*) V_1(z_1, z_0) V_2(z_2, z_0) f(\sigma_1, \sigma_2) \left(1 + O\left((\log D)^{-\frac{1}{6}}\right)\right) + R^-(z_0, z_1, z_2).
		\end{align}
		Here $V(z_0, h^*), V_1(z_1, z_0)$, and $V_2(z_2, z_0)$ are given by \eqref{definition_Vi},
		\begin{align}\label{Eq: Function F sigma1 sigma2}
			F(\sigma_1, \sigma_2)= \inf \left\{ F_1(s_1) F_2(s_2)\colon  \frac{s_1}{\sigma_1} + \frac{s_2}{\sigma_2} = 1, \, s_i \geq \beta_i - 1 \right\},
		\end{align}
		\begin{align}\label{Eq: Function f sigma1 sigma2}
			f(\sigma_1, \sigma_2)= \sup \left\{ f_1(s_1) F_2(s_2) + f_2(s_2) F_1(s_1) - F_1(s_1) F_2(s_2) : \frac{s_1}{\sigma_1} + \frac{s_2}{\sigma_2} = 1, \, s_i \geq \beta_i \right\},
		\end{align}
		where $f_1, f_2, F_1, F_2$ as in Lemma \ref{Lem: Beta sieve}, and for any \( \varepsilon > 0 \),
		\[
		R^+(z_0, z_1, z_2), R^-(z_0, z_1, z_2) \ll_{\varepsilon} \sum_{\substack{d_1d_2\leq D^{1+\varepsilon}\\ d_1 \mid P_1(z_1)\\ d_2 \mid P_2(z_2)}} \tau(d_1 d_2)^4 |r(d_1, d_2)|.\]
	\end{prop}

	\begin{proof}
		Our proof will resemble Heath-Brown and Li \cite[Section 3.3]{HBL2016}. So, we will only give the key changes in the preliminary sieving procedure. We will apply the beta sieves to establish the result. Before that, we perform pre-sieving to handle the small prime factors less than $z_0$. We proceed by writing \[
		P(z_ 0) :=\prod_{p<z_0} p .
		\]
		Throughout the proof, we let $\mathbf{d}=(d_1,\: d_2)$. We begin with the pre-sieving step.
		
		Suppose that it is given that $(d_1d_2,\: P(z_0))=1$. For \( (m, n)\in  \mathcal{W}_{\mathbf{d}} \), let \( \mathcal{U}(\mathbf{d}) \) denote the set of products \( m'n' \), where \( m' \) and \( n' \) are the largest divisors of \( m \) and \( n \) such that all prime factors of \( m' \) and \( n' \) belong to \( \mathcal{P}_1 \) and \( \mathcal{P}_2 \), respectively. The elements of \( \mathcal{W}_{\mathbf{d}} \) are counted according to their multiplicity.

		First, we wish to estimate the sum
		\begin{align*}
			S(\mathcal{U}(\mathbf{d}), z_0):=\sum_{\substack{k\in \mathcal{U}(\mathbf{d})\\ (k,\: P(z_0))=1}}1.
		\end{align*}
		We will use the Fundamental Lemma of the sieve methods to estimate the above sum. To apply it, first, we verify the sieve axioms.
		
		Note that if \( d \mid	P(z_ 0  ) \), by axiom \ref{axiom 1}, we have
		\[
		\#\mathcal{U}(\mathbf{d}) = h^*(d) h(\mathbf{d}) X + r(d),
		\]
		where
		\[
		h^*(d) = \sum_{\substack{d = e_1 e_2 e_3\\ e_1e_3| P_1(z_0)\\e_2e_3| P_2(z_0)}} h(e_1 e_3, e_2 e_3) \mu(e_3)
		\]
		and
		\[
		r(d) = \sum_{\substack{d = e_1 e_2 e_3\\ e_1e_3| P_1(z_0)\\e_2e_3| P_2(z_0)}} r(d_1 e_1 e_3, d_2 e_2 e_3) \mu(e_3).
		\]
		In particular, by axiom \ref{axiom 2}, we have
		\begin{equation}\label{definition_h*}
			h^*(p) =
			\begin{cases} 
				h(p, 1) + h(1, p) - h(p, p) \leq \frac{2C_1}{p}, & \text{if } p \in \mathcal{P}_1 \cap \mathcal{P}_2, \\[6pt]
				h(p, 1) \leq \frac{C_1}{p}, & \text{if } p \in \mathcal{P}_1 \setminus \mathcal{P}_2, \\[6pt]
				h(1, p) \leq \frac{C_1}{p}, & \text{if } p \in \mathcal{P}_2 \setminus \mathcal{P}_1, \\[6pt]
				0, & \text{otherwise}.
			\end{cases}
		\end{equation}
		
		We can now apply the Fundamental Lemma of the sieve of level $x^s$ \cite[Theorem 6.9]{FI2010}  to \( \mathcal{U}(\mathbf{d}) \), with $s$ to be chosen later to optimize the error terms to obtain
		\begin{align}\label{fundamental}
			S(\mathcal{U}(\mathbf{d}), z_0) &= h(\mathbf{d}) X V(z_0, h^*) 
			+ O_{C_1}(h(\mathbf{d}) X e^{-s}) \nonumber \\
			&\quad + O\Bigg(\sum_{\substack{e_1 e_2 e_3 < z_0^s \\ e_1 e_2 e_3 \mid P(z_0) \\ e_1 e_3 \mid P_1(z_0), \, e_2 e_3 \mid P_2(z_0)}} 
			|r(d_1 e_1 e_3, d_2 e_2 e_3)|\Bigg).
		\end{align}
		
		Next, we define
		\[
		P_i(z_0, z) := \prod_{\substack{z_0 \leq p < z \\ p \in \mathcal{P}_i}} p \quad \text{for $i\in \{1, 2\}$}.
		\]
		Let 
		\[
		\mathcal{W}^{*}: = \{(m, n) \in \mathcal{W} : (m, P_1(z_0)) = (n, P_2(z_0)) = 1\}.
		\]
		
		Now, we are in a position to introduce the beta sieves. For \( i\in \{1, 2\} \), let \( (\lambda_i^\pm(d))_{d\leq D_i}\) be upper and lower bound beta sieve weights of level $D_i = z_i^{s_i}$, where \( \beta_i - 1 \leq s_i \ll 1 \).  

        
        We begin with the upper bound. By the definition of upper bound sieve, we have
		\begin{align*}
			S(\mathcal{W}; \mathbf{z}) &= \sum_{(m, n) \in \mathcal{W}^{*}} 1_{(m, P_1(z_0, z_1))=1} 1_{(n, P_2(z_0, z_2))=1}\\ 
			&\leq \sum_{(m, n) \in \mathcal{W}^{*}} \bigg( \sum_{d_1 \mid (m, P_1(z_0, z_1))} \lambda_1^+(d_1) \bigg)
			\bigg( \sum_{d_2 \mid (n, P_2(z_0, z_2))} \lambda_2^+(d_2) \bigg) \\ 
			&= \sum_{d_1 \mid P_1(z_0, z_1)} \sum_{d_2 \mid P_2(z_0, z_2)} \lambda_1^+(d_1) \lambda_2^+(d_2) \#\mathcal{W}^{*}_{\mathbf{d}}.
		\end{align*}
		Note that
		\[
		\#\mathcal{W}^{*}_{\mathbf{d}} = S(\mathcal{U}(\mathbf{d}), z_0).
		\]
		Therefore, we may apply the estimate from \eqref{fundamental} to obtain
		\[
		S(\mathcal{W}; \mathbf{z}) \leq X V(z_0, h^*) \Sigma + O(E_1) + O(E_2),
		\]
		where
		\[
		\Sigma = \sum_{d_1 \mid P_1(z_0, z_1)} \sum_{d_2 \mid P_2(z_0, z_2)} \lambda_1^+(d_1) \lambda_2^+(d_2) h(\mathbf{d}),
		\]
		and the error terms are given by 
		\[
		E_1 = X e^{-s} \sum_{\substack{d_1 < D_1 \\ d_1 \mid P_1(z_1)}} \sum_{\substack{d_2 < D_2 \\ d_2 \mid P_2(z_2)}} h(\mathbf{d}),
		\]
		and
		\[
		E_2 = \sum_{\substack{f_1 \leq D_1 z_0^s \\ f_1 \mid P_1(z_1)}} \sum_{\substack{f_2 \leq D_2 z_0^s \\ f_2 \mid P_2(z_2)}} \tau^2(f_1) \tau^2(f_2) |r(f_1, f_2)|.
		\]
		One can apply axiom \ref{axiom 3} and proceed to estimate the sum \( \Sigma \) and the error terms \( E_1 \) and \( E_2 \) as in \cite[Section 3.3]{HBL2016} by  choosing \[
		s = (\log z_1  z_2)^{1/3}.\]

		For the lower bound, we cannot directly multiply the lower bound sieve weights for $1_{(m,\: P_1(z_0, z_1))=1}$ and $1_{(n,\: P_2(z_0, z_2))=1}$ since for some $m$ and $n$, both lower bounds might be negative. To circumvent this issue, we employ the usual idea in the vector sieve; for example, see \cite[Lemma 10.1]{Har2007}. For brevity, we write
        \begin{align*}
            \delta_1^\pm(m) :=\sum_{d_1|(m,\: P_1(z_0, z_1)) }\lambda_1^\pm(d_1) \quad \text{and} \quad  \delta_2^\pm(n) :=\sum_{d_2|(n,\: P_2(z_0, z_2)) }\lambda_2^\pm(d_2).
        \end{align*}
        Then, by definition of sieves, we have
        \begin{align*}
          \delta_1^-(m)\leq   1_{(m,\: P_1(z_0, z_1))=1}\leq \delta_1^+(m)
        \quad \text{and} \quad
          \delta_2^-(n)\leq   1_{(n,\: P_2(z_0, z_2))=1}\leq \delta_2^+(n).
        \end{align*}
        Applying the above inequalities, we have
        \begin{align*}
            1_{(m,\: P_1(z_0, z_1))=1} 1_{(n,\: P_2(z_0, z_2))=1} &=1_{(m,\: P_1(z_0, z_1))=1}\delta_2^+(n)-1_{(m,\: P_1(z_0, z_1))=1} \big(\delta_2^+(n)-1_{(n,\: P_2(z_0, z_2))=1}\big)\\
            &\geq \delta_1^-(m)\delta_2^+(n)-\delta_1^+(m)\big(\delta_2^+(n)-1_{(n,\: P_2(z_0, z_2))=1}\big)\\
            &\geq \delta_1^-(m)\delta_2^+(n)+\delta_1^+(m)\delta_2^-(n)-\delta_1^+(m)\delta_2^+(n).
        \end{align*}
        Therefore, we have
		\begin{align*}
			S(\mathcal{W}; \mathbf{z}) &= \sum_{(m, n) \in \mathcal{W}^{*}} 1_{(m,\: P_1(z_0, z_1))=1} 1_{(n,\:  P_2(z_0, z_2))=1} \\ 
			&\geq  \sum_{d_1 \mid P_1(z_0, z_1)} \sum_{d_2 \mid P_2(z_0, z_2)} \lambda_1^-(d_1) \lambda_2^+(d_2) \#\mathcal{W}^{*}_{\mathbf{d}}  + \sum_{d_1 \mid P_1(z_0, z_1)} \sum_{d_2 \mid P_2(z_0, z_2)} \lambda_1^+(d_1) \lambda_2^-(d_2) \#\mathcal{W}^{*}_{\mathbf{d}} \\ 
			&\quad - \sum_{d_1 \mid P_1(z_0, z_1)} \sum_{d_2 \mid P_2(z_0, z_2)} \lambda_1^+(d_1) \lambda_2^+(d_2) \#\mathcal{W}^{*}_{\mathbf{d}}.
		\end{align*}
		One can now proceed similarly to \cite[Section 3.3]{HBL2016} to get the desired lower bound. This completes the proof of the proposition.
	\end{proof}

	\section{Sieve estimates}\label{sec: sieve estimates}

	Recall from \eqref{Eq: App of Buchstab} and \eqref{Eq: Main estimate} that if  \( \theta_1, \theta_2, \theta, \lambda \) are such that
	\[
	\lambda>0, \quad 0 < \theta_2 < \theta_1 < \frac{1}{2}, \quad \theta_2 < \theta < 1, \quad \text{and} \quad \theta_1\asymp \theta_2,
	\] 
	then we wish to show that there exists a positive constant $c$ such that
	\begin{align}\notag 
		S(\mathcal{A}, \mathcal{Q}, x^{\theta_1}) - \sum_{\substack{x^{\theta_1}<p_1\leq x^{1/2}\\ p_1\equiv 3\Mod 4}}S(\mathcal{A}_{p_1}, \mathcal{Q}, p_1)-\lambda \sum_{x^{\theta_2} \leq p \leq x^{\theta}} w_p \#\mathcal{B}_p \geq \left(c + o(1)\right) \frac{x}{(\log x)^{5/2}},
	\end{align}
	where $S(\mathcal{A}, \mathcal{Q}, z)$ is given by \eqref{Def: sifting semi}, $\mathcal{B}_p$ is given by \eqref{Def: set B}, and 
	\[w_p=1-\dfrac{\log p}{\theta \log x}.\]
	For the convenience of the readers, we also recall that
	\begin{align*}
		\mathcal{P}=\{ p\: \text{prime}\} \quad \text{and} \quad \mathcal{Q}=\{p\equiv 3\Mod 4\},
	\end{align*}
	and for any $z\geq 3$,
	\begin{align*}
		P(z)=\prod_{\substack{p<z\\ p\in \mathcal{P}}}p \quad \text{and} \quad Q(z)=\prod_{\substack{p<z\\ p\in \mathcal{Q}}}p.
	\end{align*}

	First, we apply the lower bound vector sieve estimate \eqref{lowerbound} from Proposition \ref{vector_sieve} to obtain the following result.
	\begin{prop}\label{Prop: lower bound}
		If $0<\theta_2<\theta_1<1/2$, we have
		\begin{align*}
			S(\mathcal{A}, \mathcal{Q}, x^{\theta_1}) \geq  \left(\frac{C\kappa e^{-3\gamma/2}}{4} + o(1)\right) 
			\frac{f\left((2\theta_1)^{-1}, (2\theta_2)^{-1}\right)}{\theta_1^{1/2} \theta_2} 
			\frac{x}{(\log x)^{5/2}},
		\end{align*}
		where the function $f$ is given as in Proposition \ref{vector_sieve}, $\gamma$ is the Euler-Mascheroni constant, 
		\[ C = \frac{9}{4} \prod_{\substack{ p > 3\\p \equiv 3 \pmod{4}}} \left(1 - \frac{3p - 1}{(p-1)^3}\right) 
		\prod_{p \equiv 1 \pmod{4}} \left(1 - \frac{1}{(p-1)^2}\right).
		\]
		and
		\[\kappa=\bigg(\dfrac{\pi}{2}\bigg)^{1/2}\prod_{p\equiv 3\pmod {4}}\bigg(1-\dfrac{1}{p^2}\bigg)^{1/2}.\]
	\end{prop}
	
	\begin{proof}
		We will apply the lower-bound vector sieve to the set 
		\begin{align}\label{Def: W}
			\mathcal{W}=\big\{(p-1, p+2)\colon 3\leq p\leq x-2, p\equiv 3\Mod 8\big\}
		\end{align}	
		by choosing $\mathcal{P}_1=\mathcal{Q}$, $\mathcal{P}_2=\mathcal{P}$, $\mathbf{z}=(x^{\theta_1}, x^{\theta_2})$ in Proposition \ref{vector_sieve}.
		First, we note that
		\begin{align}\label{Eq: lower i}
			S(\mathcal{A}, \mathcal{Q}, x^{\theta_1})\geq S\big(\mathcal{W}; (x^{\theta_1}, \: x^{\theta_2})\big).
		\end{align}
		For brevity, we set 
		\[z_1=x^{\theta_1} \quad \text{and} \quad  z_2=x^{\theta_2}.\]
		
		First, we define \( h\colon \mathbb{N}^2\to (0,1] \) as follows
		\begin{align}\label{Eq: lower def h}
			h(d_1, d_2) =
			\begin{cases}
				\frac{1}{\varphi(d_1 d_2)}, & \text{if } (d_1, d_2) = (d_1, 2) = (d_2, 2) = 1, \\[6pt]
				\frac{2}{\varphi(d_1)\varphi(d_2)}, & \text{if } (d_1, d_2) = 3, \, (d_1, 2) = (d_2, 2) = 1, \\[6pt]
				0, & \text{otherwise}.
			\end{cases}
		\end{align}
		Then, if $\mathbf{d}=(d_1, d_2)$, we have
		\[
		\#\mathcal{W}_{\mathbf{d}} = h(\mathbf{d}) \frac{\pi(x)}{\varphi(8)} + r(\mathbf{d}),
		\]
		where \( r(\mathbf{d}) \) is an error term.
		Indeed, by Lemma \ref{Lemma_HBL_2}, for any $\varepsilon>0$ and $A\geq 3$, we have
		\[
		\sum_{\substack{\mathbf{d} \\ d_1 d_2 < x^{1/2-\varepsilon}}} \tau(d_1 d_2)^4 |r(\mathbf{d})| \ll 
		\frac{x}{(\log x)^A}.
		\]
		Therefore, we choose $D=x^{1/2-2\varepsilon}$ in Proposition \ref{vector_sieve} to obtain
		\begin{align}\label{Eq: lower ii}
			S\big(\mathcal{W}; (z_1, \: z_2)\big) \geq \dfrac{\pi(x)}{4}V(z_0, h^*)V_1(z_1, z_0)V_2(z_2, z_0)f((2\theta_1)^{-1}, (2\theta_2)^{-1})(1+o(1)) + \dfrac{x}{(\log x)^{10}},
		\end{align}
		where $z_0=e^{(\log z_1z_2)^{1/3}}$. Next, we wish to simplify the above expression. We begin by examining the asymptotics of the factor \( V(z_0, h^*) V_1(z_1, z_0) V_2(z_2, z_0) \). 
		Recall that by \eqref{Def: hstar}, we have 
		\[  h^*(d)=\sum_{\substack{e_1e_2e_3=d\\ e_1e_3|P_1(z_0)\\ e_2e_3|P_2(z_0)}}h(e_1e_3, e_2e_3)\mu(e_3).\]
		So, by \eqref{Eq: lower def h}, we have 
		\[
		h^*(p) =
		\begin{cases}
			h(p, 1) + h(1, p) - h(p, p) =
			\begin{cases} 
				\frac{2}{\varphi(p)}, & \text{if } p \neq 3, \\[8pt]
				\frac{1}{2}, & \text{if } p = 3,
			\end{cases} & \text{if } p \equiv 3 \Mod 4, \\[12pt]
			h(p, 1) = \frac{1}{\varphi(p)}, & \text{if } p \equiv 1 \Mod 4 \text{ or } p = 2.
		\end{cases}
		\]
		Moreover, by \eqref{definition_Vi} and \eqref{Eq: lower def h}, we have  
		\[
		V_1(z_1, z_0) = \prod_{\substack{z_0 \leq p < z_1 \\ p \equiv 3 \Mod 4}} \left(1 - h(p, 1)\right) = \prod_{\substack{z_0 \leq p < z_1 \\ p \equiv 3 \Mod 4}} \left(1 - \frac{1}{\varphi(p)}\right),
		\]
		\[
		V_2(z_2, z_0) = \prod_{z_0 \leq p < z_2} \left(1 - h(1, p)\right) = \prod_{z_0 \leq p < z_2} \left(1 - \frac{1}{\varphi(p)}\right).
		\]
		Therefore, the Euler factors in \( V(z_0, h^*) V_1(z_1, z_0) V_2(z_2, z_0) \) are given as follows
		\[
		\begin{cases}
			1 = 2\left(1 - \frac{1}{2}\right), & \text{if } p = 2, \\[8pt]
			\frac{1}{2} = \frac{9}{8} \left(1 - \frac{1}{3}\right)^2, & \text{if } p = 3, \\[8pt]
			1 - \frac{2}{p-1} = \left(1 - \frac{3p - 1}{(p-1)^3}\right)\left(1 - \frac{1}{p}\right)^2, & \text{if } 3 < p < z_0 \text{ and } p \equiv 3 \Mod 4, \\[8pt]
			1 - \frac{1}{p-1} = \left(1 - \frac{1}{(p-1)^2}\right)\left(1 - \frac{1}{p}\right), & \text{if } 3 \leq p < z_0 \text{ and } p \equiv 1\Mod 4, \\[8pt]
			\left(1 - \frac{1}{p-1}\right)^2 = \left(1 - \frac{1}{(p-1)^2}\right)^2\left(1 - \frac{1}{p}\right)^2, & \text{if } z_0 \leq p < z_2 \text{ and } p \equiv 3 \Mod 4, \\[8pt]
			1 - \frac{1}{p-1} = \left(1 - \frac{1}{(p-1)^2}\right)\left(1 - \frac{1}{p}\right), & \text{if } z_0 \leq p < z_2 \text{ and } p \equiv 1 \Mod 4, \\[8pt]
			1 - \frac{1}{p-1} = \left(1 - \frac{1}{(p-1)^2}\right)\left(1 - \frac{1}{p}\right), & \text{if } z_2 \leq p < z_1 \text{ and } p \equiv 3 \Mod 4.
		\end{cases}
		\]
		Thus, we have
		\begin{equation}\label{VV_1V_2_asymptotic}
			V(z_0, h^*) V_1(z_1, z_0) V_2(z_2, z_0) \sim C V_1(z_1) V_2(z_2),
		\end{equation}
		where
		\begin{align}\label{Def: C}
			C = \frac{9}{4} \prod_{\substack{ p > 3\\p \equiv 3 \Mod 4}} \left(1 - \frac{3p - 1}{(p-1)^3}\right) 
			\prod_{p \equiv 1 \Mod 4} \left(1 - \frac{1}{(p-1)^2}\right),
		\end{align}
		and
		\begin{align}
			V_1(z_1) &= \prod_{\substack{p < z_1 \\ p \equiv 3 \Mod 4}} \left(1 - \frac{1}{p}\right) \sim \kappa \left(\frac{e^{-\gamma}}{\log z_1}\right)^{1/2} = \frac{\kappa e^{-\gamma/2} }{\theta_1^{1/2} (\log x)^{1/2}}, \label{definition_V_1}\\
			V_2(z_2) &= \prod_{p < z_2} \left(1 - \frac{1}{p}\right) \sim \frac{e^{-\gamma}}{ \log z_2} = \frac{e^{-\gamma}} {\theta_2 \log x},\label{definition_V_2}
		\end{align}
		using the Mertens estimate and recalling that $z_1=x^{\theta_1}$ and $z_2=x^{\theta_2}$. Finally, we combine \eqref{Eq: lower i}, \eqref{Eq: lower ii}, \eqref{VV_1V_2_asymptotic}, \eqref{definition_V_1}, and \eqref{definition_V_2} to obtain
		\begin{align}\label{term_1}
			S(\mathcal{A}, \mathcal{Q}, x^{\theta_1}) &\geq \big(C + o(1)\big) \frac{\pi(x)}{4} V_1(z_1) V_2(z_2) 
			f\left((2\theta_1)^{-1}, (2\theta_2)^{-1}\right) \nonumber \\
			&\sim \left(\frac{C\kappa e^{-3\gamma/2}}{4}+ o(1)\right) 
			\frac{f\left((2\theta_1)^{-1}, (2\theta_2)^{-1}\right)}{\theta_1^{1/2} \theta_2} 
			\frac{x}{(\log x)^{5/2}},
		\end{align}
		as desired.
	\end{proof}

	We now incorporate the sieve ingredients differently by using a ``switching principle'' due to Iwaniec \cite{Iwa1972} to give an upper bound in the following proposition.
	
	\begin{prop}\label{Prop: upper bound}
		Let $1/4<  \theta_1\leq 1/2$ and $0<\theta_2<\theta_1$. Then, for sufficiently large \( x \), we have
		\begin{align*}
			\sum_{\substack{x^{\theta_1} < p_1 \leq \sqrt{x} \\ p_1 \equiv 3 \Mod 4}} S(\mathcal{A}_{p_1}, \mathcal{Q}, p_1) 
			\leq \left(2 c_1 c_2^2 c_3 C(\theta_1) + o(1)\right) \frac{x}{\theta_2\min(\theta_1, 1/2-\theta_2)(\log x)^{5/2}},
		\end{align*}
		where
		\[
		c_1 := \prod_{p > 3} \left(1 + \frac{1}{(p-2)^2}\right), \quad
		c_2 := \prod_{p > 2} \left(1 - \frac{1}{(p-1)^2}\right), 
		\]	
		\[
		c_3 := \frac{2}{\sqrt{\pi}} 
		\prod_{p \equiv 1 \Mod 4}\left(1 - \frac{1}{p}\right)^{1/2} 
		\left(1 + \frac{p^2}{(p-1)^3}\right) 
		\prod_{p \equiv 3 \Mod 4} \left(1 - \frac{1}{p}\right)^{1/2},
		\]
		and
		\[
		C(\theta_1) := \int_0^{1-2\theta_1} \frac{1}{\beta^{1/2}(1-\beta)} 
		\log\left(\frac{1-\beta-\theta_1}{\theta_1}\right) d\beta.
		\]
	\end{prop}

	\begin{proof}

		For brevity, let us write
		\[
		T := \sum_{\substack{x^{\theta_1}< p_1 \leq \sqrt{x} \\ p_1 \equiv 3 \pmod{4}}} S(\mathcal{A}_{p_1}, \mathcal{Q}, p_1).
		\]
		Since \( \theta_1>1/4 \), it follows that any number counted in the sum \( T \) must have exactly one prime factor congruent to \(3 \pmod{4} \) in addition to \( p_1 \), and exactly one factor of \( 2 \). This is because the number of prime factors \( \equiv 3 \pmod{4} \) in \( p-1 \) is even, and there cannot be four of them. Hence, the numbers counted in \( T \) take the form
		\begin{equation}\label{T_term}
			p - 1 = 2mp_1p_2 \quad \text{and} \quad  (p+2, P(z_2)) = 1,
		\end{equation}
		where
		\[
		p' \mid m \implies p' \equiv 1 \pmod{4}, \quad z_1 < p_1 <p_2 \leq \sqrt{x},  \quad p_1 \equiv p_2 \equiv 3 \Mod 4.
		\]
		These conditions imply that
		\[
		m < x/(2z_1^2) \quad \text{and} \quad z_1 < p_1 < (x/2m)^{1/2},
		\]
		where we will denote $z_1=x^{\theta_1}$ throughout the proof.
		
		Before proceeding, we introduce a notation. Given any positive integer $n$, we write $b^*(n)=1$ to denote that all the prime factors of $n$ are congruent to $1\Mod 4$, and $b^*(n)=0$ otherwise.
		
		Let \( (\lambda_1^+(d))_{d\leq D_1}\) and \( (\lambda_2^+(d))_{d\leq D_2} \) be two linear upper bound sieves and of levels \( D_1 \leq  z_1\) and \( D_2\leq z_2 \), respectively, to be chosen later. Then, we have
		\begin{align*}
			1_{2mp_1p_2+1\: \text{prime}}\leq \sum_{d_1|2mp_1p_2+1}\lambda_1^+(d_1) \quad \text{and}\quad 1_{(2mp_1p_2+3, \: P(z_2))=1}\leq \sum_{d_2|2mp_1p_2+3}\lambda_2^+(d_2).
		\end{align*}
		This implies that
		\begin{align}
			T & \leq (1+o(1))\sum_{m \leq x/(2z_1^2)} b^*(m) 
			\sum_{\substack{z_1 < p_1 \leq \sqrt{\frac{x}{2m}} \\ p_1 \equiv 3\Mod 4}}
			\sum_{\substack{3<p_2 \leq \frac{x}{2mp_1} \\ p_2 \equiv 3 \Mod 4}}
			\sum_{d_1 \mid (2mp_1p_2 + 1)} \lambda_1^+(d_1) 
			\sum_{d_2\mid (2mp_1p_2 + 3)} \lambda_2^+(d_2) \nonumber \\
			&=(1+o(1)) \sum_{m \leq x/(2z_1^2)} b^*(m) \sum_{\substack{d_1 \leq D_1\\ d_2 \leq D_2\\ (d_1, 2m) = (d_2, 6m) = 1\\(d_1, d_2) = 1}} \lambda_1^+(d_1) \lambda_2^+(d_2)
			\sum_{\substack{z_1< p_1 \leq \sqrt{\frac{x}{2m}} \\ p_1 \equiv 3 \Mod 4}}
			\sum_{\substack{3<p_2 \leq \frac{x}{2mp_1} \\ p_2 \equiv 3 \Mod 4\\ 
					2mp_1p_2 + 1 \equiv 0 \Mod {d_1} \\ 
					2mp_1p_2 + 3 \equiv 0 \Mod {d_2}}}1,\label{inequality_T}
		\end{align}
		where we have used the fact that $d_1|2mp_1p_2+1$ and $d_2|2mp_1p_2+3$, so we have $(d_1, 2)=(d_2, 2)=1$ and $(d_1, d_2)=1$ or $2$. Since both $d_1$ and $d_2$ are odd, we can conclude that $(d_1, d_2)=1$. Moreover, $(d_2, 3)=1$ since $d_2|2mp_1p_2+3$. 
		
        We now want to estimate the inner sum over $p_2$ in the above expression. To do that, for any $y\geq 2$, we define 
		\[
		E(y; m, q, a) := \pi(y; m, q, a) - \frac{\pi(y; m, 1, 1)}{\varphi(q)},
		\]	
		where $\pi(y; m, q, a)$ is given by \eqref{Def: generalized prime counting function}. Then, we rewrite \eqref{inequality_T} as
		\begin{align}
			\notag 	T \leq &\:  (1+o(1)) \sum_{m \leq \frac{x}{2z_1^2}} b^*(m) 
			\sum_{\substack{d_1 \leq D_1\\ d_2 \leq D_2 \\ (d_1, 2m) = (d_2, 6m) =1\\ (d_1, d_2) = 1}} 
			\lambda_1^+(d_1) \lambda_2^+(d_2) 
			\sum_{\substack{z_1 < p_1 \leq \sqrt{\frac{x}{2m}} \\ p_1 \equiv 3 \Mod 4}}
			\pi\left(x; 2mp_1, 4d_1d_2, r\right) \\
			\label{Eq: Sum T} 	=&\:  (1+o(1)) \sum_{m \leq \frac{x}{2z_1^2}} b^*(m)  
			\sum_{\substack{d_1\leq D_1\\ d_2 \leq D_2\\ (d_1, 2m) = (d_2, 6m) =1\\ (d_1, d_2) = 1}} 
			\lambda_1^+(d_1) \lambda_2^+(d_2) 
			\sum_{\substack{z_1 < p_1 \leq \sqrt{\frac{x}{2m}} \\ p_1 \equiv 3 \Mod 4}}
			\frac{\pi\left(x; 2mp_1, 1, 1\right)}{\varphi(4d_1d_2)} + O(E(x)),
		\end{align}
		for some reduced residue class \( r \) modulo \( 4d_1d_2 \), where 
		\[
		E(x) \ll \bigg| \sum_{m \leq \frac{x}{2z_1^2}} b^*(m) 
		\sum_{\substack{d_1 \leq D_1\\ d_2\leq D_2 \\ (d_1, 2m) = (d_2, 6m) =1\\ (d_1, d_2) = 1}} 
		\lambda_1^+(d_1) \lambda_2^+(d_2) 
		\sum_{\substack{z_1 < p_1 \leq \sqrt{\frac{x}{2m}} \\ p_1 \equiv 3 \Mod 4}}
		E\left(x; 2mp_1, 4d_1d_2, r\right)\bigg|.
		\]
		First, we show that the contributions from $E(x)$ is negligible. By the Cauchy-Schwarz inequality, we have
		\begin{align}
			E (x)\ll & \sum_{\substack{d_1 \leq D_1\\ d_2 \leq D_2 \\ (d_1, d_2) = 1}} 
			\Bigg| \sum_{\substack{m \leq \frac{x}{2z_1^2}\\ (d_1, 2m) = (d_2, 6m) = 1}} b^*(m) 
			\sum_{\substack{z_1 < p_1 \leq \sqrt{\frac{x}{2m}} \\ p_1 \equiv 3 \Mod 4}}
			E\left(x; 2mp_1, 4d_1d_2, r\right) \Bigg| \nonumber \\
			\ll & \sum_{q \leq 4D_1D_2} \tau(q) 
			\Bigg| \sum_{\substack{m \leq \frac{x}{2z_1^2}\\ (q, 6m)=1}} b^*(m) 
			\sum_{\substack{z_1 < p_1 \leq \sqrt{\frac{x}{2m}} \\ p_1 \equiv 3 \Mod 4}}
			E\left(x; 2mp_1, q, r\right) \Bigg|\nonumber \\
			\ll & \left( \sum_{q \leq 4D_1D_2} \tau(q)^2 
			|E(x,q)| \right)^{1/2} \left( \sum_{q \leq 4D_1D_2} 
			|E(x,q)| \right)^{1/2},\label{Cauchy-Schwarz}
		\end{align}
		where \[
		E(x,q):= \sum_{\substack{m \leq \frac{x}{2z_1^2}\\ (6m, q)=1}} b^*(m) 
		\sum_{\substack{z_1 < p_1 \leq \sqrt{\frac{x}{2m}} \\ p_1 \equiv 3 \Mod 4}}
		E\left(x; 2mp_1, q, r\right).
		\]
		First, we use the trivial bound \( E(x; n, q, r) \ll \frac{x}{n \varphi(q)} \) and the fact that $\varphi(n)\gg n/\log \log n$ together with $\sum_{q\leq x}\tau(q)^2/q\ll (\log x)^4$ to note that
		\begin{align}
			\notag 	\left( \sum_{q \leq 4D_1D_2} \tau(q)^2 |E(x, q)| \right)^{\frac{1}{2}} 
			&\ll \left( \sum_{n \leq x} \sum_{q \leq x} \tau(q)^2 |E(x; n, q, a)| \right)^{\frac{1}{2}} \\
			\notag 	&\ll \left( \sum_{n \leq x} \sum_{q \leq x} \tau(q)^2 \frac{x \log \log q}{n q} \right)^{\frac{1}{2}}\\
			\label{Cauchy-Schwarz i}	&\ll \left( x (\log x)^6 \right)^{1/2}.
		\end{align}
		Next, for the second sum in \eqref{Cauchy-Schwarz}, for any \( A > 0 \), suppose
		\begin{equation}\label{condition_D_L}
			4D_1D_2 \leq \frac{(x/2)^{1/2}}{(\log(x/2))^B}, \quad B= \frac{3(2A+6)}{2} + 17.
		\end{equation}
        Now, we write
        \begin{align*}
            \mathcal{L}=\{\ell\colon \ell=2mp_1, b^*(m)=1, m\leq x/z_1^2, (6m, q)=1, z_1<p_1\leq (x/2m)^{1/2}, p_1\equiv 3\Mod 4\}.
        \end{align*}
		Then by noting that 
		\[
		2mp_1 \leq x^{1-\theta_1}
		\]
        so that we have $\#\mathcal{L}\ll x^{1-\theta_1}$. Therefore, we apply Theorem \ref{theorem_pan} with $f=1_{\mathcal{L}}$ to obtain
		\begin{align}\label{Cauchy-Schwarz ii}
			\sum_{q \leq 4D_1D_2} |E(x, q)| 
			\ll  \frac{x}{(\log x)^{2A+6}}.
		\end{align}
		Therefore, from \eqref{Cauchy-Schwarz}, \eqref{Cauchy-Schwarz i}, and \eqref{Cauchy-Schwarz ii}, we have \[
		E(x)\ll \frac{x}{(\log x)^A},
		\]
		which is an acceptable error term. Hence, from \eqref{Eq: Sum T}, we have
		\begin{align}
			T  \leq & \: (1+o(1))	\sum_{m \leq \frac{x}{2z_1^2}} b^*(m)  \sum_{\substack{d_1 \leq D_1\\ d_2 \leq D_2 \\ (d_1, 2m) = (d_2, 6m) =1\\ (d_1, d_2) = 1}} 
			\lambda_1^+(d_1) \lambda_2^+(d_2) 
			\sum_{\substack{z_1< p_1 \leq \sqrt{\frac{x}{2m}} \\ p_1 \equiv 3 \Mod 4}}
			\frac{\pi\left(x;2 mp_1, 1, 1\right)}{\varphi(4d_1d_2)} \nonumber\\
			= &\: (1+o(1)) \sum_{m \leq \frac{x}{2z_1^2}} b^*(m) \sum_{\substack{d_1 \leq D_1\\ d_2\leq D_2 \\ (d_1, 2m) = (d_2, 6m) =1\\ (d_1, d_2) = 1}} 
			\frac{\lambda_1^+(d_1) \lambda_2^+(d_2)}{2\varphi(d_1) \varphi(d_2)} 
			\sum_{\substack{z_1< p_1 \leq \sqrt{\frac{x}{2m}} \\ p_1 \equiv 3 \Mod 4}}
			\pi\left(\frac{x}{2mp_1}\right), \label{upper_bound_T}
		\end{align}
		where $\pi(t)$ counts the number of primes not exceeding $t$ for any $t\geq 2$.
		
		Now, we wish to estimate the sum over the sieve weights. Since \( \lambda_1^+ \) and \( \lambda_2^+ \) are the upper bound linear sieve weights, we may apply \cite[Theorem 5.9]{FI2010} to obtain 
		\[
		\sum_{\substack{d_1 \leq D_1\\ d_2 \leq D_2 \\ (d_1, 2m) = (d_2, 6m) =1\\ (d_1, d_2) = 1}} 
		\frac{\lambda_1^+(d_1) \lambda_2^+(d_2)}{\varphi(d_1) \varphi(d_2)} 
		\leq  \frac{\big(4 + o(1)\big) C_1 H_1(m) H_2(m)}{(\log D_1)( \log D_2)},
		\]
		where 
		\[
		C_1 = \prod_{(p, 6m) = 1} \left(1 + \frac{1}{(p-2)^2}\right) \leq \prod_{p > 3} \left(1 + \frac{1}{(p-2)^2}\right),
		\]
		\[
		H_1(m) = \frac{2m}{\varphi(2m)}\prod_{(p, 2m) = 1} \left(1 - \frac{1}{(p-1)^2}\right)  \leq \frac{2m}{\varphi(2m)} \prod_{p > 2} \left(1 - \frac{1}{(p-1)^2}\right),
		\]
		\[
		H_2(m) = \frac{6m}{\varphi(6m)}\prod_{(p, 6m) = 1} \left(1 - \frac{1}{(p-1)^2}\right)  \leq \frac{6m}{\varphi(6m)} \prod_{p > 3} \left(1 - \frac{1}{(p-1)^2}\right).
		\]
		For brevity, we denote
		\[
		c_1:= \prod_{p > 3} \left(1 + \frac{1}{(p-2)^2}\right) \quad \text{and} \quad 
		c_2 := \prod_{p > 2} \left(1 - \frac{1}{(p-1)^2}\right).
		\]
		Then, if $(m, 6)=1$, we have
		\[\sum_{\substack{d_1 \leq D_1\\ d_2 \leq D_2 \\ (d_1, 2m) = (d_2, 6m) =1\\ (d_1, d_2) = 1}} 
		\frac{\lambda_1^+(d_1) \lambda_2^+(d_2)}{\varphi(d_1) \varphi(d_2)} \leq \frac{(32 + o(1)) c_1 c_2^2 m^2}{\varphi^2(m)(\log D_1) (\log D_2)}.
		\]
		Combining the above estimate together with \eqref{upper_bound_T}, we have
		\begin{align}
			\notag 	T \leq & \; (16 + o(1)) \frac{c_1 c_2^2}{(\log D_1) (\log D_2)}\sum_{m \leq \frac{x}{2z_1^2}} \frac{b^*(m)m^2}{\varphi^2(m)} 	\pi\left(\frac{x}{2mp_1}\right) \\
			\label{Eq: T i}		= & \;
			(8 + o(1)) \frac{c_1 c_2^2 x}{(\log D_1) (\log D_2)} 	
			\sum_{m \leq \frac{x}{2z_1^2}} \frac{b^*(m)m}{\varphi^2(m)} 
			\sum_{\substack{z_1 < p_1 \leq \sqrt{\frac{x}{2m}} \\ p_1 \equiv 3 \Mod 4}}
			\frac{1}{p_1 \log \big(\frac{x}{2mp_1}\big)} ,		
		\end{align}
		where we have used the Prime Number Theorem in the last line.

		To estimate the inner sum, we define the functions: for any $t, u\geq 2$,
		\[
		h(t) := \frac{1}{\log \left(\frac{x}{2mt}\right)}, \quad H(u) := \int_{z_1}^{\sqrt{\frac{x}{2m}}} \frac{1}{\log \left(\frac{x}{2ut}\right)} d(\log \log t).
		\]
		Recall that by Mertens' Theorem, for any $t\geq 4$, there exists a constant $C_0>0$ such that
		\[
		S_1(t) := \sum_{\substack{p < t \\ p \equiv 3 \Mod 4}} \dfrac{1}{p} = \frac{1}{2} \log \log t + C_0 + O\left(\frac{1}{\log t}\right).
		\]
		Applying integration by parts to the inner sum, we obtain
		\begin{align*}
			\sum_{\substack{z_1 < p_1 \leq \sqrt{\frac{x}{2m}} \\ p_1 \equiv 3 \pmod{4}}} 
			\frac{1}{p_1 \log \big(\frac{x}{2mp_1}\big)} &= \int_{z_1}^{\sqrt{\frac{x}{2m}}} h(t) \, d\left(\frac{1}{2} \log \log t\right) + O\left(\frac{h\left(\sqrt{\frac{x}{2m}}\right)}{\log z_1}\right) \\
			&= \frac{1}{2} H(m) + O\left(\frac{1}{(\log x)^2}\right).
		\end{align*}
		Therefore, we have
		\begin{equation}\label{outer_sum}
			\sum_{m \leq \frac{x}{2z_1^2}} \frac{b^*(m)m}{\varphi^2(m)} 
			\sum_{\substack{z_1 < p_1 \leq \sqrt{\frac{x}{2m}} \\ p_1 \equiv 3 \pmod{4}}}
			\frac{1}{p_1 \log \big(\frac{x}{2mp_1}\big)}=	\frac{1}{2} 	\sum_{m \leq \frac{x}{2z_1^2}} \frac{b^*(m) m}{\varphi(m)^2} H(m) + O\bigg(\frac{1}{(\log x)^2}\sum_{m \leq \frac{x}{2z_1^2}} \frac{b^*(m) m}{\varphi(m)^2} \bigg).
		\end{equation}
		For $t\geq 2$, we let
		\begin{align*}
			S_2(t) := \sum_{m \leq t} \frac{b^*(m) m}{\varphi(m)^2} \intertext{and for any integer $n\geq 1$, we let}
			g(n): = \frac{b^*(n)n^2}{\varphi(n)^2}.
		\end{align*}
		First, we note that since $\varphi(n)\gg n/(\log \log n)$, we have 
		\[
		g(n) \ll (\log \log n)^2.
		\]
		Next, by Mertens' theorem, for any $w\geq 2$, we have
		\begin{align*}
			\sum_{p \leq w} \frac{g(p) \log p}{p} 
			&= \sum_{\substack{p \leq w \\ p \equiv 1 \Mod 4}} \frac{p \log p}{(p-1)^2} \\
			&= \sum_{\substack{p \leq w \\ p \equiv 1 \Mod 4}} \frac{\log p}{p} 
			+ \sum_{\substack{p \leq w \\ p \equiv 1 \Mod 4}} \frac{(2p-1) \log p}{p(p-1)^2}\\
			&=\frac{1}{2} \log w + O(1).
		\end{align*}
		Therefore, the conditions in \cite[Proposition A.3.]{Mar2002} are satisfied, and so, we have\footnote{We could have applied a variant of Wirsing's theorem \cite{Wir1961}. However, for convenience, we have chosen to use \cite[Proposition A.3.]{Mar2002}}
		\begin{align}\label{Eq: sum S2}
			S_2(t) = \sum_{n \leq t} \frac{g(n)}{n} = c_3 (\log t)^{1/2} + O\left(\frac{1}{(\log t)^{1/2}}\right),
		\end{align}
		where
		\[
		c_3 = \frac{2}{\sqrt{\pi}} 
		\prod_{p \equiv 1 \Mod 4} \left(1 - \frac{1}{p}\right)^{1/2} 
		\left(1 + \frac{p^2}{(p-1)^3}\right) 
		\prod_{p \equiv 3 \Mod 4} \left(1 - \frac{1}{p}\right)^{1/2}.
		\]
		Therefore, using \eqref{Eq: sum S2} in \eqref{outer_sum}, we have
		\begin{align}\label{outer_sum ii}
			\sum_{m \leq \frac{x}{2z_1^2}} \frac{b^*(m)m}{\varphi^2(m)} 
			\sum_{\substack{z_1 < p_1 \leq \sqrt{\frac{x}{2m}} \\ p_1 \equiv 3 \pmod{4}}}
			\frac{1}{p_1 \log \big(\frac{x}{2mp_1}\big)}=	\frac{1}{2}\sum_{m \leq \frac{x}{2z_1^2}} \frac{b^*(m) m}{\varphi(m)^2} H(m) + O\left(\frac{1}{(\log x)^{3/2}}\right).
		\end{align}
		Next, we evaluate the main term in the above expression. Indeed, using \eqref{Eq: sum S2}, we have
		\begin{align}
			\notag 	\frac{1}{2}\sum_{m \leq \frac{x}{2z_1^2}} \frac{b^*(m) m}{\varphi(m)^2} H(m)
			&\leq \frac{1}{2} \int_1^{\frac{x}{2z_1^2}} H(u) \, dS_2(u) \\
			\notag 	&= \frac{c_3 }{2} \int_1^{\frac{x}{2z_1^2}} H(u) \, d(\log u)^{1/2} 
			+ O\left(\frac{\max \left(H(1), H\left(\frac{x}{2z_1^2}\right)\right)}{(\log \frac{x}{2z_1^2})^{1/2}}\right) \\
			\label{Eq: Main i}	&= \frac{c_3 }{2} \int_1^{\frac{x}{2z_1^2}} H(u) \, d(\log u)^{1/2} 
			+ O\left(\frac{\log \log x}{(\log x)^{3/2}}\right).
		\end{align}

		Let \[\varepsilon_1 = \frac{\log z_1}{\log \frac{x}{2}}.  \]
		Note that $\varepsilon_1\sim \theta_1$ as $z_1=x^{\theta_1}$ by our assumption. To evaluate the integral, we make the change of variables \(t = (x/2)^{\alpha}\), \(u = (x/2)^{\beta}\) to obtain
		\begin{align}
			\notag 	\int_1^{\frac{x}{2z_1^2}} H(u) \, d(\log u)^{1/2} 
			&= \int_1^{\frac{x}{2z_1^2}} \int_{z_1}^{\sqrt{\frac{x}{2u}}} \frac{1}{\log \left(\frac{x}{2ut}\right)} \, d(\log \log t) \, d(\log u)^{1/2} \\	
			\notag 	&= \frac{1}{2\left(\log \frac{x}{2}\right)^{1/2}} \int_0^{1-2\varepsilon_1} \frac{1}{\beta^{1/2}} 
			\int_{\varepsilon_1}^{\frac{1-\beta}{2}} \frac{1}{\alpha(1-\alpha-\beta)} \, d\alpha \, d\beta \\	
			\label{Eq: Main ii}	&= \frac{1}{2\left(\log \frac{x}{2}\right)^{1/2}} \int_0^{1-2\varepsilon_1} \frac{1}{\beta^{1/2}(1-\beta)} 
			\log\left(\frac{1-\beta-\varepsilon_1}{\varepsilon_1}\right) \, d\beta.
		\end{align}
		We write 
		\[
		C(\varepsilon_1) := \int_0^{1-2\varepsilon_1} \frac{1}{\beta^{1/2}(1-\beta)} \log\left(\frac{1-\beta-\varepsilon_1}{\varepsilon_1}\right) d\beta.
		\]
		Therefore, we can infer from \eqref{Eq: T i}, \eqref{outer_sum ii}, \eqref{Eq: Main i}, \eqref{Eq: Main ii} that		\[
		T \leq \left(2 c_1 c_2^2 c_3 C(\varepsilon_1) + o(1)\right) \frac{x}{(\log D_1) (\log D_2)\left(\log \frac{x}{2}\right)^{1/2}}.
		\]
		Finally, to minimize \( (\log D_1\log D_2)^{-1} \) under the condition \eqref{condition_D_L}, we choose  
		\[
		D_1 = \frac{(x/2)^{\min(\theta_1, 1/2-\theta_2)}}{4(\log (x/2))^{B_2}}, \quad D_2 = \left(\frac{x}{2}\right)^{\theta_2},
		\]
		which gives  
		\[
		\frac{1}{(\log D_1)( \log D_2)} \sim \frac{1}{\theta_2\min(\theta_1, 1/2-\theta_2)(\log x)^2}.
		\]
		Since \( C(\varepsilon_1) \sim C(\theta_1) \), it follows that  
		\begin{equation}\label{term_2}
			T \leq \left(2 c_1 c_2^2 c_3 C(\theta_1) + o(1)\right) \frac{x}{\theta_2\min(\theta_1, 1/2-\theta_2)(\log x)^{5/2}},
		\end{equation}
        as desired.
	\end{proof}

	For the weighted terms in \eqref{Eq: Main estimate}, we will apply Proposition \ref{vector_sieve} to obtain the following upper bound.

	\begin{prop}\label{Prop: weighted}
		Let \( z_2 = x^{\theta_2} \), \( y = x^\theta \), and \( 0 < \theta_2 < \theta < 1 \). Then, for sufficiently large \( x \), we have
		\begin{align*}
			\sum_{z_2 \leq p \leq y} w_p \#\mathcal{B}_p\leq \left(\frac{C\kappa e^{-3\gamma/2}}{4 \theta_1^{1/2} \theta_2} + o(1)\right) 
			I(\theta, \theta_1, \theta_2) \frac{x}{(\log x)^{5/2}},
		\end{align*}
		where $\gamma$ is the Euler-Mascheroni constant,
		\[
		C = \frac{9}{4} \prod_{\substack{p > 3 \\ p \equiv 3\Mod 4}} \left(1 - \frac{3p - 1}{(p-1)^3}\right) 
		\prod_{p \equiv 1 \Mod 4} \left(1 - \frac{1}{(p-2)^2}\right),
		\]
		\[
		\kappa = \bigg(\dfrac{\pi}{2}\bigg)^{1/2}\prod_{p\equiv 3\pmod {4}}\bigg(1-\dfrac{1}{p^2}\bigg)^{1/2},
		\]
		and 
		\begin{align*}
			I(\theta, \theta_1, \theta_2):=\int_{\theta_2}^\theta \left(1 - \frac{\alpha}{\theta}\right) F\left(\frac{1 - 2\alpha}{2\theta_1}, \frac{1 - 2\alpha}{2\theta_2}\right) \frac{ d\alpha}{\alpha},
		\end{align*}
		where the function \( F \) is given by \eqref{Eq: Function F sigma1 sigma2}.
	\end{prop}

	\begin{proof}
		We now apply the upper bound vector sieve to the set \(\mathcal{W}\) as defined in \eqref{Def: W} by choosing $\mathcal{P}_1=\mathcal{Q}$, $\mathcal{P}_2=\mathcal{P}$ in Proposition \ref{vector_sieve}.
		Let \( d_1 \mid P_1(z_1) = Q(z_1) \), \( d_2 \mid P_2(z_2) = P(z_2) \), and \( z_2 \leq p \leq y \), where \( z_1 = x^{\theta_1} \) and \( z_2 = x^{\theta_2} \). If \( (p, d_1) = 1 \), then 
        \[
              h(d_1, p d_2) = \frac{1}{\varphi(p)} h(d_1, d_2).
        \]
		Additionally, note that if \( p \mid d_1 \), then \( \#\mathcal{W}_{(d_1, p d_2)} = 0 \).
		Also, recall that in \eqref{VV_1V_2_asymptotic}, we have calculated the asymptotic
		\[
		V(z_0, h^*) V_1(z_1, z_0) V_2(z_2, z_0) \sim C V_1(z_1) V_2(z_2),
		\]
		where $C$ is given by \eqref{Def: C}, and \( V_1(z_1) \) and \( V_2(z_2) \) are given by \eqref{definition_V_1} and \eqref{definition_V_2}, respectively.
		
		We choose $D=x^{1/2-2\varepsilon}/p$ in Proposition \ref{vector_sieve}. Then, we apply the upper bound vector sieve \eqref{upperbound} and Lemma \ref{Lemma_HBL_2} to obtain 
		\[
		\#\mathcal{B}_p \leq (C + o(1)) \frac{\pi(x)}{4(p-1)} V_1(z_1) V_2(z_2) F(s_1(p), s_2(p)) 
		+ O\bigg(\dfrac{x/p}{(\log x)^{10}}\bigg),
		\]
		where 
		\[
		s_i(p) = \frac{\log(x^{1/2-2\varepsilon}/p)}{\log z_i} \quad \text{for $i\in \{1, 2\}$}.
		\]
		Therefore, by Mertens' estimate,
		\begin{align}\label{Eq: weight i}
			\sum_{z_2 \leq p \leq y} w_p \#\mathcal{B}_p \leq \left(\frac{C}{4} + o(1)\right) \pi(x) V_1(z_1) V_2(z_2) 
			\sum_{z_2 \leq p \leq y} \frac{w_p F(s_1(p), s_2(p))}{p-1} + O\bigg(\dfrac{x}{(\log x)^{10}}\bigg).
		\end{align}
		Next, we write
		\[
		\sigma_i(t) = \frac{\log(\sqrt{x}/t)}{\log z_i} \quad \text{for $i\in \{1, 2\}$}.
		\]
		Then, by the Prime Number Theorem and by a change of variable $t=x^\alpha$, we have
		\begin{align*}
			\sum_{z_2 \leq p \leq y} \frac{w_p F(s_1(p), s_2(p))}{p-1} 
			&\leq \int_{z_2}^y \left(1 - \frac{\log t}{\log y}\right) F(\sigma_1(t), \sigma_2(t)) d(\log \log t) \\
			&= \int_{\theta_2}^\theta \left(1 - \frac{\alpha}{\theta}\right) F\left(\frac{1 - 2\alpha}{2\theta_1}, \frac{1 - 2\alpha}{2\theta_2}\right) \frac{d\alpha}{\alpha}\\
			& := I(\theta, \theta_1, \theta_2),
		\end{align*}
		say. Hence, by \eqref{Eq: weight i} together with \eqref{definition_V_1} and \eqref{definition_V_2}, we conclude that
		\begin{equation}\label{term_3}
			\sum_{z_2 \leq p \leq y} w_p \#\mathcal{B}_p \leq \left(\frac{C\kappa e^{-3\gamma/2}}{4\theta_1^{1/2 }\theta_2} + o(1)\right) I(\theta, \theta_1, \theta_2) \frac{x}{(\log x)^{5/2}}.
		\end{equation}
	\end{proof}

	\section{Proof of Theorem \ref{theorem1.1}}\label{sec: proof}
	First, we will apply Lemma \ref{Lem: Beta sieve} with $\kappa_1=1/2$ and $\kappa_2=1$. Indeed, by \cite[p. 225]{FI2010}, we have  $\beta_1(1/2)=1$ and $\beta_2(1)=2$. So, we can explicitly compute the functions $F_i, f_i$ in Lemma \ref{Lem: Beta sieve} and Proposition \ref{vector_sieve}. 
	
	Indeed, if $\beta_1=1$ in Proposition \ref{vector_sieve}, by Lemma \ref{Lem: Beta sieve} and \cite[(14.2)--(14.3)]{FI2010}, we have
	\[
	\begin{aligned}
		F_1(s) &= 2 \left(\frac{e^\gamma}{\pi s}\right)^{1/2}, \quad &0 < s \leq 2, \\
		f_1(s) &= \left(\frac{e^\gamma}{\pi s}\right)^{1/2} \log\left(1 + 2(s - 1) + 2\sqrt{s(s - 1)}\right), \quad &1 \leq s \leq 3,
	\end{aligned}
	\]
	and continues with  the differential equations
	\begin{align*}		
		&\dfrac{d}{ds}(F_1(s)) = \frac{f_1(s - 1) - F_1(s)}{2s}, &  s > 2, \\[6pt]
		&\dfrac{d}{ds}(f_1(s)) = \frac{F_1(s - 1) - f_1(s)}{2s}, &  s > 1.
	\end{align*}

	Next, if $\beta_2=2$ in Proposition \ref{vector_sieve}, by Lemma \ref{Lem: Beta sieve} and \cite[(12.1)--(12.2)]{FI2010}, we have
	\begin{align*}
		F_2(s) &= \frac{2e^\gamma}{s}, \quad & 1 \leq s \leq 3, \\
		f_2(s) &= \frac{2e^\gamma \log(s - 1)}{s}, \quad &2 \leq s \leq 4,
	\end{align*}
	and continues with the differential equations 
	\begin{align*}		
		&\dfrac{d}{ds}(F_2(s)) = {f_2(s - 1) - F_2(s)}, &  s > 3, \\[6pt]
		&\dfrac{d}{ds}(f_2(s)) = {F_2(s - 1) - f_2(s)}, &  s > 2.
	\end{align*}
Moreover, recall from \eqref{Eq: Function F sigma1 sigma2} and \eqref{Eq: Function f sigma1 sigma2}, we can express the functions \( F \) and \( f \) as
	\[
	F(\sigma_1, \sigma_2) := \inf \bigg\{ F_1(s_1) F_2(s_2) : \frac{s_1}{\sigma_1} + \frac{s_2}{\sigma_2} = 1, \, s_1 >0, \,  s_2\geq 1 \bigg\}.
	\]
	and 
	\begin{equation}
	\begin{aligned}\label{Eq: f proof}
		f(\sigma_1, \sigma_2) &:= \sup \bigg\{ f_1(s_1)F_2(s_2) + f_2(s_2)F_1(s_1) - F_1(s_1)F_2(s_2) :  \frac{s_1}{\sigma_1} + \frac{s_2}{\sigma_2} = 1, \, s_1 \geq 1, \,  s_2\geq 2 \bigg\}.
	\end{aligned}
	\end{equation}
	
	After recollecting the above information, we are now ready to complete the proof of Theorem \ref{theorem1.1}. To avoid confusion, we mention that we will use notations from Section \ref{sec: outline} without further comment.

	\begin{proof}[Proof of Theorem \ref{theorem1.1}]
		Let
		\begin{align}\label{Eq: Conditions}
			0 < \theta_2 < \theta_1 < \frac{1}{2}, \quad \theta_1 > \frac{1}{4}, \quad \theta_2 < \theta < 1, \quad \lambda > 0.
		\end{align}
		We apply Propositions \ref{Prop: lower bound}, \ref{Prop: upper bound}, and \ref{Prop: weighted} to deduce that
		\[
		S(\mathcal{A}, \mathcal{Q}, x^{1/2}) - \lambda \sum_{x^{\theta_2} \leq p \leq x^{\theta}} w_p \#\mathcal{B}_p \geq H(\lambda, \theta, \theta_1, \theta_2) \frac{x}{(\log x)^{5/2}},
		\]
		where \( H(\lambda, \theta, \theta_1, \theta_2) \) is given by
		\[
		H(\lambda, \theta, \theta_1, \theta_2) := \frac{C\kappa e^{{-3\gamma}/{2}}}{4} \cdot \frac{f\big((2\theta_1)^{-1}, (2\theta_2)^{-1}\big)}{\theta_1^{1/2} \theta_2} 
		- \frac{2 c_1 c_2^2 c_3 C(\theta_1)}{\theta_2\min(\theta_1, 1/2-\theta_2)} 
		- \lambda\frac{C\kappa e^{-{3\gamma}/{2}}}{4 \theta_1^{1/2} \theta_2} I(\theta, \theta_1, \theta_2).
		\]
		Here, $f$ is given by \eqref{Eq: f proof}, $\gamma$ is the Euler-Mascheroni constant, the constant terms are given by
		\[
		C := \frac{9}{4} \prod_{\substack{p > 3 \\ p \equiv 3 \Mod 4}} \left(1 - \frac{3p - 1}{(p - 1)^3}\right) \prod_{p \equiv 1 \Mod 4} \left(1 - \frac{1}{(p - 1)^2}\right),
		\]
		\[
		c_1 := \prod_{p > 3} \left(1 + \frac{1}{(p - 2)^2}\right),
		\quad c_2 := \prod_{p > 2} \left(1 - \frac{1}{(p - 1)^2}\right),\]
		\[
		c_3 := \frac{2}{\sqrt{\pi}} 
		\prod_{p \equiv 1 \Mod 4} \left(1 - \frac{1}{p}\right)^{1/2} 
		\left(1 + \frac{p^2}{(p-1)^3}\right) 
		\prod_{p \equiv 3 \Mod 4} \left(1 - \frac{1}{p}\right)^{1/2},
		\]
		\[
		\kappa=\bigg(\dfrac{\pi}{2}\bigg)^{1/2}\prod_{p\equiv 3\pmod {4}}\bigg(1-\dfrac{1}{p^2}\bigg)^{1/2},
		\]
        and the integral terms are given by
		\[
		C(\theta_1) := \int_0^{1 - 2\theta_1} \frac{1}{\beta^{1/2}(1 - \beta)} \log\left(\frac{1 - \beta - \theta_1}{\theta_1}\right) d\beta,
		\]
		\[
		I(\theta, \theta_1, \theta_2) := \int_{\theta_2}^\theta \left(1 - \frac{\alpha}{\theta}\right) F\left(\frac{1 - 2\alpha}{2\theta_1}, \frac{1 - 2\alpha}{2\theta_2}\right) \frac{d\alpha}{\alpha}.
		\]
		Recall from \eqref{number_of_divisors_p+2}, the prime divisors of \( p+2 \) satisfies
		\[\omega(p+2) \leq \dfrac{1}{\lambda} + \dfrac{1}{\theta}.
		\]  
		So, given \eqref{Eq: Conditions}, we aim to find the smallest possible value of \( 1/\lambda + 1/\theta \) such that \[ H(\lambda, \theta, \theta_1, \theta_2) > 0 .\]
		Additionally, to compare the improvement in the number of divisors of \( p+2 \) by including the weighted term, we define the function \( G(\theta_1, \theta_2) \) by considering only the first two terms in \( H(\lambda, \theta, \theta_1, \theta_2) \):
		\[
		G(\theta_1, \theta_2) := \frac{C\kappa e^{-{3\gamma}/{2}}}{4} \cdot \frac{f\big((2\theta_1)^{-1}, (2\theta_2)^{-1}\big)}{\theta_1^{1/2} \theta_2} 
		- \frac{2 c_1 c_2^2 c_3 C(\theta_1)}{\theta_2\min(\theta_1, 1/2-\theta_2)}.
		\]

		First, we use Matlab\footnote{Matlab files with computation are included with this work on arxiv.org.} to perform a stepwise iterative search for the values \(\theta_1\) and \(\theta_2\) such that \(G(\theta_1, \theta_2) > 0\) and \(\theta_2\) is as large as possible, we find that the largest possible \(\theta_2\), with a step size of 0.0001, is \(\theta_2 =  0.0249\). Taking \(\theta_1 =  0.4230\), we obtain
		\[
		G(\theta_1, \theta_2) \approx 0.001378	.
		\]
		This corresponds to the result that \(\omega(p+2) < 1/\theta_2 <40.16\) without the weighted term \(\lambda \sum_{x^{\theta_2} \leq p \leq y} w_p \#\mathcal{B}_p\).
		
		Next, including the weighted term \(\lambda \sum_{x^{\theta_2} \leq p \leq y} w_p \#\mathcal{B}_p\), we perform a rough search near the estimated values of the parameters \(\theta\), \(\lambda\), \(\theta_1\), and \(\theta_2\). We obtain
		\[
		H(0.17,0.25,0.449,0.011)
		=1.2690
		\quad 
		\text{and} 
		\quad 
		\omega(p+2) < \frac{1}{0.17} + \frac{1}{0.25}
		= 9.8824.
		\]
		Therefore, we conclude that 
		\[
		\#\left\{ 
		p : p \leq x, \, p = 1 + m^2 + n^2, \, m, n \in \mathbb{N}, \, \Omega(p+2) \leq 9
		\right\} \gg \frac{x}{(\log x)^{5/2}},
		\]
		as desired.
	\end{proof}

	\end{document}